\documentclass[10pt]{article}
\usepackage{amsmath, amssymb ,amsthm, amsfonts, amsgen}
\usepackage{graphicx}
\usepackage[dvips]{color}
\numberwithin{equation}{section} 

\makeatletter
\def\rightharpoonupfill@{\arrowfill@\relbar\relbar\rightharpoonup}
\newcommand{\xrightharpoonup}[2][]{\ext@arrow
0359\rightharpoonupfill@{#1}{#2}} \makeatother

\def\e{{\varepsilon}}

\def\f{{\varphi}}

\newtheorem{theorem}{Theorem}[section]
\newtheorem{lemma}[theorem]{Lemma}
\newtheorem{proposition}[theorem]{Proposition}
\newtheorem{corollary}[theorem]{Corollary}

\newtheorem{remark}[theorem]{Remark}
\newtheorem{definition}[theorem]{Definition}
\newtheorem{example}[theorem]{Example}
\newtheorem{notation}[theorem]{Notation}

\newcommand{{\rr}}{{\mathbb R}}

\newenvironment{@abssec}[1]{%
     \if@twocolumn
       \section*{#1}%
     \else
       \vspace{.05in}\footnotesize
       \parindent .2in
         {\upshape\bfseries #1. }\ignorespaces
     \fi}
     {\if@twocolumn\else\par\vspace{.1in}\fi}

\def\supess{\mathop{\rm ess\: sup }}
\begin{document}

\title{Existence of Minimizers for Non-Level Convex Supremal Functionals}
\author{Ana Margarida Ribeiro\thanks{
Centro de  Matem\'{a}tica e Aplica\c{c}\~{o}es (CMA) and Departamento de Matem\'{a}tica, FCT, UNL
Quinta da Torre, 2829-516 Caparica, Portugal.
E-mail: amfr@fct.unl.pt, } \;\;
Elvira Zappale \thanks{Dipartimento di Ingegneria Industriale, Universit\'{a}
degli Studi di Salerno, Via Giovanni Paolo II, 132, 84084 Fisciano (SA) Italy.
E-mail:ezappale@unisa.it}}

\maketitle

\begin{abstract}
The paper is devoted to  determine necessary and sufficient conditions for existence of solutions to the problem
$\displaystyle \inf\left\{ \operatorname*{ess\,sup}_{x\in\Omega} f(\nabla u(x)): u \in u_0 + W^{1,\infty}_0(\Omega)\right\}$, when the supremand $f$ is not necessarily level convex. These conditions are obtained through a comparison with the related level convex problem and are written in terms of a differential inclusion involving the boundary datum. Several conditions of convexity for the supremand $f$ are also investigated.\medskip

\begin{center}{\bf R\'{e}sum\'{e}}\end{center}

Dans cet article on \'{e}tudie des conditions n\'{e}cessaires et suffisantes pour l'existence de solutions pour le probl\`{e}me de minimisation $\displaystyle \inf\left\{ \operatorname*{ess\,sup}_{x\in\Omega} f(\nabla u(x)): u \in u_0 + W^{1,\infty}_0(\Omega)\right\}$ lorsque la fonction $f$ n'est pas une fonction convexe par niveaux. La strat\'{e}gie utilis\'{e}e pour obtenir ces conditions est celle de comparer ce probl\`{e}me avec son probl\`{e}me relax\'{e}. On obtient comme condition n\'{e}cessaire et suffisante une inclusion diff\'{e}rentielle sur la donn\'{e}e au bord. On \'{e}tudie aussi plusieurs conditions de convexit\'{e}.

\medskip

\noindent\textbf{Keywords}: Supremal functionals, differential inclusions, convexity.

\noindent\textbf{MSC2010 classification}:  49K21, 49J45, 26B25, 46N10.
\end{abstract}

\section{Introduction}

The direct method of the calculus of variations requires some lower semicontinuity of the functional to minimize, which, in general, is related to some notion of convexity. In the lack of this `convexity', the usual procedure is to consider the relaxed problem, related to the original one, obtained by `convexification' of the `non-convex' function. This leads in many problems to an understanding of the minimizing sequences and of the infimum to the original problem, but it doesn't ensure the problem has a solution.

In this paper we will investigate necessary and sufficient conditions for existence of solutions to
$$
(P)\qquad\inf\left\{ \operatorname*{ess\,sup}_{x\in\Omega}f\left(  \nabla u\left(
x\right)  \right):\ u\in u_0+W_{0}^{1,\infty}(\Omega)\right\},
$$
when $f$ lacks of the appropriate convexity notion. We restrict our attention to the so-called scalar case, that is $u$ is a scalar function, $u:\Omega\subset\mathbb{R}^n\longrightarrow\mathbb{R}$, $n\ge 1$, $u_0\in W^{1,\infty}(\Omega)$. It is also possible to have $u_0$ Lipschitz only defined on $\partial \Omega$, $u_0:\partial\Omega\longrightarrow\mathbb{R}$. In this case, $u_0$ shall be extended to $\Omega$ as a Lipschitz function and the study of problem $(P)$ can be done according to the choice of the Lipschitz extension.

Functionals in the $L^\infty$ form, as above, provide a realistic setting to many physical problems in a variety of contexts like nonlinear elasticity, chemotherapy or imaging. For a more detailed description see \cite{BJW-ARMA} due to Barron-Jensen-Wang.

Minimizing the functional in problem $(P)$ appears also as a generalization of the Lipschitz extension problem (this is the case when $f=|\cdot|$ and $u_0$ is a given Lipschitz function defined on $\partial\Omega$) and was intensively studied by Aronsson in the 1960's,  cf. \cite{A1,A2,A3, A4}, also developing a theory on absolute minimizers, see also the monograph of Aronsson-Crandall-Juutinen \cite{ACJ} and the references therein. The problem of existence and uniqueness of Lipschitz extension has been addressed by many authors with different tools, cf. for instance \cite{J, ACJS, AS} among a wide literature. In recent years also other questions of the calculus of variations, like lower semicontinuity, relaxation, homogenization, $L^p$ approximations, dimensional reduction, $\Gamma$-convergence and supremal representation, have been addressed for $L^\infty$ functionals by several authors: Barron-Liu \cite{BL}, Barron-Jensen-Wang \cite{BJW99}, Acerbi-Buttazzo-Prinari \cite{ABP}, Briani-Garroni-Prinari \cite{BGP}, Bocea-Nesi \cite{BN}, Prinari \cite{Pr02}, \cite{Pr09}, Cardialaguet-Prinari \cite{CP}, Babadjian-Prinari-Zappale \cite{BPZ}, Zappale \cite{Z}.

The functional defined in problem $(P)$ is known to be lower semicontinuous with respect to the weak* topology of $W^{1,\infty}$ (cf. \cite{BJW99} and \cite{BJ}, see also Theorems \ref{level convexity sufficient} and \ref{level convexity necessary}), if and only if $f$ is a level convex function, that is the level sets of $f$ are convex (see Definition \ref{def1.4}). This notion is usually known in areas like convex analysis, optimization or economics as {\sl quasiconvexity}. However, since {\sl quasiconvexity} has a different meaning in the calculus of variations, we prefer to use the present terminology.

Our main interest is the case in which $f$ is not necessarily a level convex function. In this context we establish necessary and sufficient conditions for existence of solutions to $(P)$. This is done through a differential inclusion which is obtained in turn relating the original problem $(P)$ and the relaxed one
$$
(P^{\rm lc})\qquad\inf\left\{ \operatorname*{ess\,sup}_{x\in\Omega}f^{\rm lc}\left(  \nabla u\left(
x\right)  \right):\ u\in u_0+W_{0}^{1,\infty}(\Omega)\right\},
$$
where $f^{\rm lc}$ denotes the level convex envelope of $f$ (cf. Definition \ref{def1.4}). This was the procedure applied to problems in the integral form
$$ \displaystyle{\inf\left\{ \int_\Omega f\left(  \nabla u\left(
x\right)  \right)dx:\ u\in u_0+W_{0}^{1,\infty}(\Omega)\right\},}$$ and we refer to Cellina \cite{CellinaNecessary}, \cite{CellinaSufficient} and Friesecke \cite{Friesecke} in the scalar case and to Dacorogna-Marcellini \cite{DM95} and Dacorogna-Pisante-Ribeiro \cite{DPR} in the vectorial one.
For further references, see also \cite[Chapter 11]{D}.

In the present context of $L^\infty$ functionals, to our knowledge very few is known, and although our problem is a scalar one, our approach is close to the one used in the vectorial case by Dacorogna-Marcellini \cite{DM95} and Dacorogna-Pisante-Ribeiro \cite{DPR} for integral functionals. Moreover, our results are sharp since differential inclusions in the scalar case are better understood than in the vectorial one. In particular, we characterize existence of solutions to problem $(P)$ when the boundary datum is affine, say $u_0:=u_{\xi_0}$, with gradient $\xi_0 \in \mathbb R^n$ , in terms of suitable  `level convexity' properties of the relaxed density $f^{\rm lc}$ around $\xi_0$.

The paper is organized as follows. Section \ref{levelconvexitysection} is devoted to the study of the properties of a {\sl level convex} function and the {\sl level convex envelope} of a function. This includes supremal Jensen's inequality, Carath\'eodory type results and several notions of {\sl strict level convexity}, which are explored in view of uniqueness results for the minimizing problems with affine boundary data. Moreover, this may also have an independent interest for optimization purposes. Some prerequisites concerning differential inclusions are also recalled in this section.

In Section \ref{SectionRelaxation} we state the relaxation result which will be one of the key results to achieve our necessary and sufficient condition for existence. In particular, under suitable hypothesis, we show in Corollary \ref{Plc=P} that
$$\inf (P)=\inf (P^{\rm lc}).$$

Necessary and sufficient conditions for the existence of solutions to problem $(P)$ are provided and discussed in Section \ref{NecSufCond}. Our main general result is stated as Theorem \ref{NSC} and it establishes that a necessary and sufficient condition for existence of solutions to $(P)$ is $$\exists\ u\in u_0+W_0^{1,\infty}(\Omega):\ f(\nabla u(x))\le \inf (P^{\rm lc}),\ a.e.\ x\in\Omega.$$ Moreover, making use of well known results on differential inclusions, a sufficient condition to this last one can be written as $$\nabla u_0(x)\in L_{\inf (P^{\rm lc})}(f)\cup {\rm int}\, L_{\inf(P^{\rm lc})}(f^{\rm lc}),\ a.e.\ x\in\Omega,$$ where $L_c(g)$ denotes the set of level $c$ of the function $g$, that is
$$L_c(g):=\left\{\xi\in\mathbb{R}^n:\ g(\xi)\le c\right\}.$$ Then, with this characterization in mind, we explore both sufficient and necessary conditions. Regarding sufficient conditions we consider both the cases $u_0$ is an affine function or not. In particular, we can always get existence of solutions in dimension $n=1$, cf. Corollary \ref{existencedim1}, and for arbitrary dimension $n$, and arbitrary data $u_0$, if we require some regularity on the solution of the relaxed problem $(P^{\rm lc})$, together with some constant properties on $f^{\rm lc}$, we can ensure existence of solutions to $(P)$, see Theorem \ref{Theoremexistenceregular}. The constant hypothesis on $f^{\rm lc}$ will be clarified later on Theorem \ref{analog to thm 11.26 Dacorogna}, where necessary and sufficient conditions to have a solution to $(P)$ with affine boundary datum $u_0(x)=<\xi_0,x>+c$, $\xi_0\in\mathbb{R}^n$, will be explored in terms of the set $\displaystyle\{\xi\in\mathbb{R}^n:\ f^{\rm lc}(\xi)=f^{\rm lc}(\xi_0)\}$.

Concerning necessary conditions we follow the ideas of Marcellini \cite{Marcellini83}, Dacorogna-Marcellini \cite{DM95}, and Dacorogna-Pisante-Ribeiro \cite{DPR}. Our approach is done through uniqueness of solutions to a level convex problem of type $(P)$. This can be achieved if the function $f$ is strictly level convex (cf. Definition \ref{sLC}). However, as we observe, it is not reasonable to assume $f^{\rm lc}$ to satisfy such property and, mimicking Dacorogna-Marcellini \cite{DM95}, we introduce the notion of strict level convexity of a function $f$ at a point $\xi_0\in\mathbb{R}^n$ in at least one direction as: for some $\alpha\in\mathbb{R}^n\setminus \{0\}$
$$\left.\begin{array}{l} \xi_0=t\gamma+ (1-t)\eta,\ t\in(0,1)\vspace{0.2cm} \\
f(\xi_0)= \max\{f(\gamma), f(\eta)\} \end{array} \right\}\ \Longrightarrow \ <\gamma-\eta, \alpha>=0.$$
This condition turns out to be a sufficient one for uniqueness of solution to level convex problems with affine boundary datum, see Theorem \ref{thm strict level conv}. As a consequence, in Corollary \ref{corollary slc in at least}, if $f^{\rm lc}(\xi_0) < f(\xi_0)$, we prove that  $f^{\rm lc}$ satisfies the above condition if and only if the original problem $(P)$ does not admit any solution.

Finally, in the Appendix, we briefly address convexity notions in the supremal setting for the vectorial case. In \cite{BJW99}, it was investigated the right notion to ensure lower semicontinuity of the supremal functionals in the vectorial case, together with supremal notions of polyconvexity and rank one convexity. Our goal here is to clarify the relations between these notions.

\section{Level Convexity  and Differential Inclusions}\label{levelconvexitysection}

All through the paper we will use the following notation for affine functions. Given a vector $\xi_0\in\mathbb{R}^n$, by $u_{\xi_0}\Omega\subset\mathbb{R}^n\longrightarrow \mathbb{R}$ we denote a function such that $\nabla u_{\xi_0}(x)=\xi_0$, $a.e.\ x\in\Omega$, or equivalently
$u_{\xi_0}(x):= <\xi_0, x> + c, \hbox{ for some }c \in \mathbb R.$

\subsection{Properties of level convex functions and level convex envelopes}

In this section we establish some results on level convex functions and level convex envelopes which are well known in the usual convexity setting. The main reason to consider level convex functions here is that it is, together with the lower semicontinuity of the function, a sufficient and necessary condition to sequential weak* lower semicontinuity in $W^{1,\infty}(\Omega)$ for functionals in the supremal form (see Theorem \ref{level convexity sufficient} below due to Barron-Jensen-Wang \cite{BJW99} and Theorem \ref{level convexity necessary}).

We first recall definitions and properties on lower semicontinuity. We refer to \cite[Chapter 1]{CDA}, \cite[Chapter 3]{FL}, \cite[Section 7]{Rockafellar}.

\begin{definition}
(i) A function $f:\mathbb{R}^{n}\rightarrow[-\infty,+\infty]$ is said to be \emph{lower semicontinuous} if the level sets $L_c(f):=\left\{\xi\in\mathbb{R}^n:\ f(\xi)\le c\right\}$ are closed for every $c\in\mathbb{R}.$ Equivalently, $f$ is lower semicontinuous if it is sequentially lower semicontinuous, that is, if $$f(\xi)\le\lim\inf f(\xi_n),\ \text{for every }\ \xi_n\to \xi.$$

(ii) The lower semicontinuous envelope of a function $f:\mathbb{R}^{n}\rightarrow[-\infty,+\infty]$ is the function $\mathrm{lsc}f:\mathbb{R}^{n}\rightarrow[-\infty,+\infty]$ defined by $$\mathrm{lsc}f(\xi)=\sup\left\{g(\xi):\ g:\mathbb{R}^{n}\rightarrow[-\infty,+\infty],\ g\text{ lower semicontinuous, } g\le f\right\}.$$
\end{definition}

\begin{proposition}\label{characterization lsc}
Let $f:\mathbb{R}^{n}\rightarrow[-\infty,+\infty]$, the lower semicontinuous envelope of $f$ is a lower semicontinuous function and $$\mathrm{lsc}f(\xi)=\inf\left\{\lim\inf f(\xi_n):\ \xi_n\to \xi\right\},\ \forall\ \xi\in\mathbb{R}^n.$$ Moreover, for every $\xi\in\mathbb{R}^n$ there exists a sequence $\xi_n$ converging to $\xi$, such that $\mathrm{lsc}f(\xi)=\lim f(\xi_n)$.
\end{proposition}

Now we recall the notion of \emph{level convexity} and the related envelope. We observe that, in Convex Analysis and Operational Research, level convexity is usually referred as \emph{quasiconvexity}. We avoid here this designation because, in the Calculus of Variations, quasiconvexity is known as a different concept. For a reference in Operational Research, see \cite{Mangasarian}.

\begin{definition}\label{def1.4}
(i) A function $f:\mathbb{R}^{n}\rightarrow[-\infty,+\infty] $ is said to be \emph{level convex} if the level sets of $f$, $L_c(f)$, are convex for each $c\in\mathbb{R}.$
Equivalently a function $f:\mathbb R^n \to [-\infty, + \infty]$ is level convex if and only if for every $\xi,\eta \in \mathbb R^n$ and $t \in [0,1]$
\begin{equation}\nonumber
f(t\xi+ (1-t)\eta) \leq \max\{f(\xi), f(\eta)\}.
\end{equation}

(ii) The level convex envelope of a function $f:\mathbb{R}^{n}\rightarrow[-\infty,+\infty]$ is the function $f^{\rm lc}:\mathbb{R}^{n}\rightarrow[-\infty,+\infty]$ defined by $$f^{\rm lc}(\xi)=\sup\left\{g(\xi):\ g:\mathbb{R}^{n}\rightarrow[-\infty,+\infty],\ g\text{ level convex, } g\le f\right\}.$$

(iii) The lower semicontinuous level convex envelope of a function $f:\mathbb{R}^{n}\rightarrow [-\infty, +\infty]$ is the function $f^{\rm lslc}:\mathbb{R}^{n}\rightarrow[-\infty,+\infty]$ defined by $$\begin{array}{l}f^{\rm lslc}(\xi)=\sup\left\{g(\xi):\ g:\mathbb{R}^{n}\rightarrow[-\infty,+\infty],\ g\text{ lower semicontinuous and}\right.\vspace{0.2cm}\\ \hspace{8.3cm}\left.\text{level convex, } g\le f\right\}.\end{array}$$

\end{definition}

\begin{remark}\label{remark envelopes} (i) It is easily seen that $f^{\rm lc}$ is a level convex function and that $f^{\rm lslc}$ is a lower semicontinuous and level convex function. Therefore, we can call these envelopes respectively the greatest level convex function below $f$ and the greatest lower semicontinuous  level convex function below $f$.

(ii) It is easy to verify that $f^{\rm lslc}\le f^{\rm lc}\le f$ and $f^{\rm lslc}\le \mathrm{lsc}f\le f$.

(iii) The function $f(\xi)=-\xi^2$ defined in $\mathbb{R}$ provides an example of a function whose envelopes take the $-\infty$ value, indeed $f^{\rm lc}\equiv f^{\rm lslc}\equiv -\infty$.

(iv) In general, $f^{\rm lc}$ and $f^{\rm lslc}$ don't coincide. Indeed the characteristic function of $\mathbb{R}\setminus (0,1)$ is a level convex function, but it is not lower semicontinuous.

(v) In general a level convex function defined in $\mathbb R^n$, with $n>1$, may not be Borel measurable, in fact one may consider the characteristic function of  the complement of a convex set  which is not Borel measurable.

(vi) An equivalent formulation for $f^{\rm lslc}$ is given by Volle's envelope, $f^{c\gamma}$, introduced in \cite{V}.
\end{remark}

We establish some preliminary properties.

\begin{proposition}\label{level convexity of lsc f}  Let $f:\mathbb{R}^n\rightarrow[-\infty,+\infty]$.

(i) If $f$ is a level convex function, then the lower semicontinuous envelope of $f$ is still level convex, that is $\mathrm{lsc}f$ is level convex.

(ii) The following identity holds: $f^{\rm lslc}=\mathrm{lsc}(f^{\rm lc})$.
\end{proposition}

\begin{proof}
To obtain condition (i) we consider $\xi,\eta\in\mathbb{R}^n$ such that $\mathrm{lsc}f(\xi)\le c$ and $\mathrm{lsc}f(\eta)\le c$ for some fixed $c\in\mathbb{R}$ and we need to show that $\mathrm{lsc}f(\lambda \xi+(1-\lambda) \eta)\le c$ for every $\lambda\in[0,1]$. Using Proposition \ref{characterization lsc} and the level convexity of $f$ we get, for certain sequences $\xi_n\to \xi$ and $\eta_n\to \eta$,
$$\begin{array}{l}\mathrm{lsc}f(\lambda \xi+(1-\lambda) \eta)\le \lim\inf f(\lambda \xi_n+(1-\lambda)\eta_n)\le \lim\inf \max\{f(\xi_n),f(\eta_n)\}\vspace{0.2cm}\\ \phantom{\mathrm{lsc}f(\lambda \xi+(1-\lambda) \eta)}\le\max\{\mathrm{lsc}f(\xi),\mathrm{lsc}f(\eta)\}\le c, \end{array}$$ as desired.

To prove condition (ii) we start noticing that, since $f^{\rm lslc}\le f^{\rm lc}$, one has $f^{\rm lslc}=\mathrm{lsc}(f^{\rm lslc})\le\mathrm{lsc}(f^{\rm lc})$, where we have used the fact that $f^{\rm lslc}$ is lower semicontinuous. On the other hand, by condition (i), $\mathrm{lsc}(f^{\rm lc})$ is level convex and since it is also lower semicontinuous and below $f$ it follows that $\mathrm{lsc}(f^{\rm lc})\le f^{\rm lslc}$.
\end{proof}

Next we relate the level convexity of a function with a generalization of Jensen's inequality for the supremal setting. The proof can be found in Barron \cite[Theorem 30]{Barron}, (see also \cite[Theorem 1.2]{BJW99}, where the theorem is stated under a lower semicontinuity hypothesis).


\begin{theorem}\label{Jensensupremalscalar}
A Borel measurable function $f:\mathbb R^{n}\to \mathbb R$ is level convex if and only if it verifies the supremal Jensen's inequality:
\begin{equation}\nonumber
f\left(\int_\Omega \varphi\,d\mu\right)\leq \mu-\supess_{x\in\Omega} f(\varphi(x))
\end{equation}
for every probability measure $\mu$ on $\mathbb{R}^{d}$ supported on the open set $\Omega \subseteq \mathbb{R}^{d}$, and every $\varphi \in L^1_\mu(\Omega;\mathbb{R}^{n})$.

In particular, considering the Lebesgue measure, if $\Omega$ is a set with finite Lebesgue measure, $$f\left(\frac{1}{|\Omega|}\int_\Omega\varphi(x)\,dx\right)\le \operatorname*{ess\,sup}_{x\in\Omega}f(\varphi(x)),\ \forall\ \varphi\in L^1(\Omega;\mathbb{R}^n).$$
\end{theorem}

From Carath\'eodory's theorem, it follows the next characterization of the level convex envelope of a function. For  this characterization, under slightly different assumptions, we refer to \cite[Theorem 5.5]{BL}.

\begin{theorem}\label{Caratheodory for f lc}
Let $f:\mathbb{R}^{n}\rightarrow(-\infty,+\infty]$ be a function such that $f^{\rm lc}>-\infty$. Then
\begin{equation}\nonumber
f^{\rm lc}(\xi)=\inf\left\{\max_{1\le i\le n+1}f(\xi_i):\ \xi=\sum_{i=1}^{n+1}\lambda_i \xi_i,\ \lambda_i\ge 0,\ \sum_{i=1}^{n+1}\lambda_i=1\right\},\ \forall\ \xi\in \mathbb{R}^n.
\end{equation} Moreover, if $f$ is continuous and $\displaystyle\lim_{|\xi|\to +\infty}f(\xi)=+\infty$, then the infimum above is indeed a minimum and $f^{\rm lc}$ is a continuous function. In particular, in this case, $f^{\rm lslc}=f^{\rm lc}$.
\end{theorem}


\begin{proof}
Let $$h(\xi):=\inf\left\{\max_{1\le i\le I}f(\xi_i):\ \xi=\sum_{i=1}^{I}\lambda_i \xi_i,\ \lambda_i\ge 0,\ \sum_{i=1}^{I}\lambda_i=1,\ I\in\mathbb{N}\right\},\ \forall\ \xi\in \mathbb{R}^n.$$ We start showing that $f^{\rm lc}= h$.

Observe that, since $f^{\rm lc}$ is level convex and $f^{\rm lc}\le f$, if $\displaystyle \xi=\sum_{i=1}^{I}\lambda_i \xi_i$ for some  $\lambda_i\ge 0$ such that $\displaystyle \sum_{i=1}^{I}\lambda_i=1$, then $\displaystyle f^{\rm lc}(\xi)\le \max_{1\le i\le I}f(\xi_i)$. From this we conclude that $-\infty< f^{\rm lc}\le h$.

Now we prove that $h$ is level convex. Once this is proved, we achieve the identity $f^{\rm lc}=h$ by definition of $f^{\rm lc}$ and because $h\le f$. Let $c\in\mathbb{R}$, we need to show that $L_c(h)$ is convex. Let $\xi,\eta\in L_c(h)$ and $\lambda\in(0,1)$, we have to show that $h(\lambda \xi+(1-\lambda)\eta)\le c$. Since $h>-\infty$ we just need to show that, given $\varepsilon>0$, we can find $I\in\mathbb{N}$, $\lambda_i\ge 0$, with $\displaystyle\sum_{i=1}^{I}\lambda_i=1$ and $z_i\in \mathbb{R}^n$ such that $\displaystyle\lambda \xi+(1-\lambda)\eta=\sum_{i=1}^{I}\lambda_i z_i$ and $\displaystyle\max_{1\le i\le I}f(z_i)\le c+\varepsilon$. This follows easily from the fact that $\xi,\eta\in L_c(h)$ and thus we have that $f^{\rm lc}=h$.

Next we show that $I$ can be reduced to $n+1$ achieving the first assertion of the theorem. This follows from Carath\'eodory's theorem. Indeed, let $\xi\in\mathbb{R}^n$ and assume $\displaystyle \xi=\sum_{i=1}^I\lambda_i \xi_i$ for some $I>n+1$, $\xi_i\in\mathbb{R}^n$ and $\lambda_i\ge 0$, with $\displaystyle\sum_{i=1}^{I}\lambda_i=1$. In particular, $\xi\in \mathrm{co}\{\xi_1,\xi_2,...,\xi_I\}\subset\mathbb{R}^n$ and, by Carath\'eodory's theorem, we can write $\displaystyle \xi=\sum_{j=1}^{n+1}\mu_j \xi_{\gamma(j)}$ for some $\mu_j\ge 0$, with $\displaystyle \sum_{i=1}^{n+1}\mu_j=1$ and $\gamma:\{1,2,...,n+1\}\rightarrow\{1,2,...,I\}$ an into function. Defining $\eta_j=\xi_{\gamma(j)}$ we obviously have $\displaystyle\max_{1\le j\le n+1}f(\eta_j)\le\max_{1\le i\le I}f(\xi_i)$ which shows our goal.

Now we show the assertion of the theorem saying that the infimum is attained as a minimum. Let $\xi\in\mathbb{R}^n$ and let $\lambda_i^\nu\ge 0$, with $\displaystyle\sum_{i=1}^{n+1}\lambda_i^\nu=1$, and $\xi_i^\nu\in\mathbb{R}^n$ be such that $\displaystyle \xi=\sum_{i=1}^{n+1}\lambda_i^\nu \xi_i^\nu$, $\displaystyle f^{\rm lc}(\xi)=\lim_{\nu\to\infty}\max_{1\le i\le n+1}f(\xi_i^\nu)$. Without loss of generality we can assume $\displaystyle\max_{1\le i\le n+1}f(\xi_i^\nu)=f(\xi_1^\nu)$. By the assumption on the limit of $f$ at infinity we can reduce to the case where the sequences $\xi_i^\nu$ are bounded otherwise $f^{\rm lc}(\xi)=+\infty$, thus $f(\xi)=+\infty$ and the minimum is attained through the trivial convex combination of $\xi$: $\xi=1\cdot \xi$. In the case where the sequences $\xi_i^\nu$ are bounded, we have, up to a subsequence,  $\displaystyle\lim_{\nu\to\infty}\xi_i^\nu=\xi_i$ and $\displaystyle \lim_{\nu\to\infty}\lambda_i^\nu=\lambda_i$, for every $i=1,\dots, n+1$. Clearly $\displaystyle\sum_{i=1}^{n+1}\lambda_i=1$ and $\displaystyle\sum_{i=1}^{n+1}\lambda_i \xi_i=\xi.$ Using the continuity hypothesis on $f$, we get $\displaystyle f(\xi_1)=\lim_{\nu\to\infty}f(\xi_1^\nu)\ge\lim_{\nu\to\infty}f(\xi_i^\nu)=f(\xi_i)$ for every $i=1,...,n+1$ and $f^{\rm lc}(\xi)=f(\xi_1)$.

Finally we show that $f^{\rm lc}$ is continuous, under the continuity assumption on $f$ and its behavior at infinity. Let $\xi,\xi_\nu \in\mathbb{R}^n$ be such that $\displaystyle \lim_{\nu\to\infty} \xi_\nu=\xi$. First we show that $\displaystyle f^{\rm lc}(\xi)\le \lim\inf f^{\rm lc}(\xi_\nu)$. Without loss of generality assume $\lim\inf f^{\rm lc}(\xi_\nu)=\lim f^{\rm lc}(\xi_\nu)$. From what was already proved, we can consider, for each $\nu$, $\lambda_i^\nu\ge 0$, with $\displaystyle\sum_{i=1}^{n+1}\lambda_i^\nu=1$, and $\xi_i^\nu\in\mathbb{R}^n$ such that $\displaystyle \xi_\nu=\sum_{i=1}^{n+1}\lambda_i^\nu \xi_i^\nu$ and  $\displaystyle f^{\rm lc}(\xi_\nu)=\max_{1\le i\le n+1}f(\xi_i^\nu)$. Re-ordering if necessary the elements $\xi_i^\nu$, we can assume $\displaystyle\max_{1\le i\le n+1}f(\xi_i^\nu)=f(\xi_1^\nu)$ and thus $f^{\rm lc}(\xi_\nu)=f(\xi_1^\nu)$. We consider two cases. If, up to a subsequence, for some $i$, $\displaystyle \lim_{\nu \to \infty}|\xi_i^\nu|=+\infty$ then the desired inequality follows from the assumption that $\displaystyle \lim_{|z|\to\infty}f(z)=+\infty$. Otherwise we can write, up to a subsequence, that $\displaystyle  \lim_{\nu\to\infty} \xi_i^\nu=\xi_i$ and $\displaystyle  \lim_{\nu\to\infty} \lambda^\nu_i=\lambda_i$. Observe that, the continuity of $f$ implies that $\max_{1\le i\le n+1}f(\xi_i)=f(\xi_1)$ and thus $$\lim f^{\rm lc}(\xi_\nu)=\lim f(\xi_1^\nu)=f(\xi_1)=\max_{1\le i\le n+1}f(\xi_i)\ge f^{\rm lc}(\xi)$$ this last inequality following from the fact that $\displaystyle \xi=\sum_{i=1}^{n+1}\lambda_i \xi_i$ and the first assertion of the present theorem. To establish the continuity of $f^{\rm lc}$ on $\xi$, it remains to show that  $\displaystyle f^{\rm lc}(\xi)\ge \lim\sup f^{\rm lc}(\xi_\nu)$. Again, let's assume $\lim\sup f^{\rm lc}(\xi_\nu)=\lim f^{\rm lc}(\xi_\nu)$. Then, as before, we have $\displaystyle \xi=\sum_{i=1}^{n+1}\lambda_i \xi_i$ for some $\lambda_i\ge 0$, with $\displaystyle\sum_{i=1}^{n+1}\lambda_i=1$, $\displaystyle\sum_{i=1}^{n+1}\lambda_i=1$, $\xi_i\in\mathbb{R}^n$ and $\displaystyle f^{\rm lc}(\xi)=\max_{1\le i\le n+1}f(\xi_i)$.
Defining $\xi_i^\nu=\xi_i+\xi_\nu-\xi$, we have $\displaystyle \xi_\nu=\sum_{i=1}^{n+1}\lambda_i \xi_i^\nu$ and thus $\displaystyle f^{\rm lc}(\xi_\nu)\le \max_{1\le i\le n+1}f(\xi_i^\nu)$. Since $\displaystyle \lim_{\nu\to\infty}\xi_i^\nu=\xi_i$, the continuity of $f$ implies that $\displaystyle \lim_{\nu\to\infty}f(\xi_i^\nu)=f(\xi_i)$ and thus $$\lim f^{\rm lc}(\xi_\nu)\le \lim\sup \max_{1\le i\le n+1}f(\xi_i^\nu)=\max_{1\le i\le n+1}f(\xi_i)=f^{\rm lc}(\xi),$$ as desired.\end{proof}\medskip

In particular, we get the following characterization of the convex hulls of the level sets of a function.

\begin{corollary}\label{Corollary to Caratheodory}
Let $f:\mathbb{R}^{n}\rightarrow(-\infty,+\infty]$ be a continuous function such that $f^{\rm lc}>-\infty$ and $\displaystyle\lim_{|\xi|\to +\infty}f(\xi)=+\infty$, then $${\rm  co }\left\{\xi\in\mathbb{R}^n:\ f(\xi)\le c\right\}=\left\{\xi\in\mathbb{R}^n:\ f^{\rm lc}(\xi)\le c\right\},\ \forall\ c\in\mathbb{R}.$$
\end{corollary}

\begin{remark}
We can get the same assertion of the corollary if we assume $f$ lower semicontinuous, bounded from below and such that $\displaystyle \lim_{|\xi|\to\infty}\frac{f(\xi)}{|\xi|}=+\infty$. This is achieved by mimicking the proof but using Theorem \ref{thm1.8} below, instead of Theorem \ref{Caratheodory for f lc}.
\end{remark}

\begin{proof}
Of course the first set in the equality is included in the second one. Now let $\xi\in \mathbb{R}^n$ be such that $f^{\rm lc}(\xi)\le c$. Then, by Theorem \ref{Caratheodory for f lc}, there exist, for $1\le i\le n+1$, $\lambda_i\ge 0$, $\xi_i\in\mathbb{R}^n$, such that $\displaystyle \xi=\sum_{i=1}^{n+1}\lambda_i \xi_i,$ $\displaystyle \sum_{i=1}^{n+1}\lambda_i=1 $ and $\displaystyle f^{\rm lc}(\xi)=\max_{1\le i\le n+1}f(\xi_i)$. Therefore $\xi\in {\rm co}\,\{\xi_1,...,\xi_{n+1}\}$ and $\displaystyle \max_{1\le i\le n+1}f(\xi_i)\le c$ and thus $\displaystyle \xi\in {\rm  co }\left\{\eta\in\mathbb{R}^n:\ f(\eta)\le c\right\}.$
\end{proof}
\bigskip

In the same spirit of \cite[Theorem 4.98]{FL} for the convex setting, we can also get the following result which provides in particular a characterization of $f^{\rm lslc}$. The proof is omitted.

\begin{theorem}\label{thm1.8}
 Let $f:\mathbb{R}^n\rightarrow(-\infty,+\infty]$ be a function bounded from below such that $\displaystyle \lim_{|\xi|\to\infty}\frac{f(\xi)}{|\xi|}=+\infty$.

(i) If $f$ is lower semicontinuous, then the level convex envelope of $f$ is still lower semicontinuous, that is $f^{\rm lc}$ is lower semicontinuous.

(ii) The following identities hold:
$$\begin{array}{l}f^{\rm lslc}(\xi)=\mathrm{lsc}(f^{\rm lc})(\xi)=(\mathrm{lsc}f)^{\rm lc}(\xi)=\vspace{0.2cm}\\ \displaystyle =\min\left\{\max_{1\le i\le n+1}\mathrm{lsc}f(\xi_i):\ \xi=\sum_{i=1}^{n+1}\lambda_i \xi_i,\ \lambda_i\ge 0,\ \sum_{i=1}^{n+1}\lambda_i=1\right\},\ \forall\ \xi\in \mathbb{R}^n.\end{array}$$
\end{theorem}

\begin{remark}
We observe that, in comparison with Theorem \ref{Caratheodory for f lc},  there is no continuity hypothesis in Theorem \ref{thm1.8}, but the growth hypothesis at infinity is stronger than the one considered in Theorem \ref{Caratheodory for f lc}.
\end{remark}

\subsection{Strict level convexity}\label{Strict level convexity}
Next we introduce the notion of \emph{ strict level convexity} that we will relate later with uniqueness of solutions to minimum problems.

\begin{definition}\label{sLC}
A level convex function $f:\mathbb R^n \to [-\infty;+ \infty]$ is said to be {\it strictly level convex} if
\begin{equation}\nonumber
f(t\xi+ (1-t)\eta)< \max\{f(\xi), f(\eta)\},
\end{equation}
for every $t \in (0,1)$ and every $\xi,\eta \in \mathbb R^n$, $\xi \not = \eta$.
\end{definition}

The above definition can be given also if $f$ is defined on a convex subset of $\mathbb R^n$.

Clearly Definition \ref{sLC} is more stringent than level convexity.
In the remaining part of this subsection we give a characterization of strict level convexity, known in Convex Analysis, and that will be exploited in the sequel. We also introduce weaker conditions than strict level convexity.

We begin recalling some definitions and some results.

\begin{definition}\label{DefC.20Grippo}
 (i) A set $C$ is said to be strictly convex if for every $x,y \in \partial C$ with $x \not =y$, for every $z:= tx+ (1-t)y$, $t \in (0,1)$, $z \in {\rm int}(C)$.

(ii) A point $x$ in a convex set $C \subset \mathbb R^n$ is an extreme point   of $C$ if and only if there exists no points $y,z \in C$, both distinct from $x$ such that $x= (1-t)y + tz$ for some $t \in (0,1)$.
The set of extreme points of $C$ is called  the profile of $C$ and is denoted by $Ext(C)$.

(iii) A point $x$ in a convex set $C \subset \mathbb R^n$ is an exposed point of $C$ if and only if there exists a supporting hyperplane $H$, such that $H \cap C= \{x\}$. The set of exposed points of a convex set $C$ is denoted by $Exp(C)$.
\end{definition}

\begin{remark}\label{moreonCAandSLC}
(i) Of course $Exp(C)\subset Ext(C)$ for any convex set $C\subset\mathbb R^n$. We observe however that there may exist extreme points which are not exposed, even if, cf. \cite[Theorem 18.6]{Rockafellar}, the set $Exp(C)$  is dense in $Ext(C)$. Consider, for example, the set $$\left\{(x,y)\in\mathbb{R}^2:\ x^2+y^2\le 1\right\}\cup \left( [0,1]\times [-1,1]\right).$$

(ii) When a convex set $C$ is closed, strict convexity of $C$ is equivalent to the condition that every boundary point of $C$ is an extreme point of $C$.
\end{remark}

\begin{proposition}\label{Ex=Exp}
Let $C \subset \mathbb R^n$ be a  strictly convex set, then $Ext(C) =Exp(C)$.
\end{proposition}

\begin{remark}\label{counterexampletoEx=Exp}
The converse of Proposition \ref{Ex=Exp} is false as one can easily see considering the set $Q= [0,1]^2$ in $\mathbb R^2$, where $Ext(Q)=\{(0,0), (0,1), (1,0), (1,1)\}=Exp(Q)$ and $Q$ is not strictly convex.
\end{remark}

\begin{proof}[Proof]
Let $x \in Ext(C)$, and assume by contradiction that $x$ is not exposed. Since $x$ is not exposed, for every supporting hyperplane $H$ at $x$, it results that there exists $y\not=x$, $y \in H \cap C$. Clearly $y \in \partial C$, and $ty +(1-t)x \in H \cap C \subset \partial C$, $t \in (0,1)$  (see for instance \cite[Corollary 18.1.3]{Rockafellar}) and this contradicts the strict convexity of $C$.
\end{proof}

Now we introduce a notation.

\begin{notation}\label{CArecalls}
Let $f:\mathbb R^n \to [-\infty,+ \infty]$ and $c \in \mathbb [-\infty,+ \infty]$. We define the set  $$R_c(f): = \{x\in\mathbb{R}^n: f(x)= c\}.$$
\end{notation}

The proof of the following result can be found in \cite[Theorem 4.3]{Danao}.

\begin{theorem}\label{Danao4.3}
Let $f$ be a real valued function defined on a convex set $C $ in $\mathbb R^n$. The function $f$ is strictly level convex if and only if for every $c$ in the range of $f$ the following conditions are verified:
\begin{itemize}
\item[(i)] $L_c(f) $ is convex,
\item[(ii)] $R_c (f) \subseteq {\rm Ext}(L_c(f))$.
\end{itemize}
\end{theorem}

The previous characterization doesn't ensure the strict convexity of $L_c(f)$. To get this property we need to assume the continuity of $f$.

\begin{proposition}\label{strictlevelconvexityandcontinuity}
Let $f$ be a real valued function defined on a convex set $C$ in $\mathbb R^n$. If $f$ is strictly level convex and continuous then for every $c$ in the range of $f$, $L_c(f)$ is closed and strictly convex for every $c \in \mathbb R$.
\end{proposition}
\begin{proof}[Proof]
The closedness of $L_c(f)$ is a consequence of the lower semicontinuity of $f$. By the level convexity of $f$ follows that $L_c(f)$ is convex. It remains to prove its strict convexity.

To this end, by Remark \ref{moreonCAandSLC} (ii), it will be enough to show that $\partial L_c(f)\subseteq Ext(L_c(f))$. By Theorem \ref{Danao4.3} it is known that $R_c(f)\subseteq Ext(L_c(f))$. Thus it will suffice to prove that
$\partial L_c(f)\subseteq R_c(f)$.

Let $y \in \partial L_c (f)$. Then, by the continuity of $f$, $f(y)\le c$. Assume $y\notin R_c(f)$, that is $f(y)< c$. Again, the continuity of $f$ would imply that $y \in {\rm int}(L_c(f))$ which is clearly a contradiction. That concludes the proof.
\end{proof}

Next we give several different characterizations of strict level convexity.

\begin{proposition}\label{carastrictlc}
Let $f:\mathbb R^n \to (-\infty,+ \infty]$ be a level convex function. Then $f$ is strictly level convex if and only if one of the following conditions is satisfied.\medskip

(i) $f(t\xi+ (1-t)\eta)= \max\{f(\xi), f(\eta)\}$ for some $t \in (0,1)$ implies $\xi=\eta$.\medskip

(ii) $\displaystyle f\left(\xi+ \frac{1}{2}\eta\right) = \max\left\{f(\xi),f(\xi+\eta)\right\}$
implies $\eta=0$.\medskip

(iii) In the case $f$ is a Borel measurable and finite function, $$\displaystyle{f\left(\int_\Omega \varphi\,d \mu\right) < \mu-\operatorname*{ess\,sup}_{x\in\Omega} f(\varphi(x))},$$
for every probability measure $\mu$ on $\mathbb{R}^n$ supported in the open set $\Omega\subset\mathbb{R}^d$ and every  nonconstant $\varphi \in L^1_\mu(\Omega;\mathbb{R}^n)$.
\end{proposition}

\begin{remark}\label{weakMqcx}
(i) In condition (ii), the value $1/2$ can be replaced by any $t\in (0,1)$.

(ii) By Proposition \ref{carastrictlc} (iii), we observe that if a Borel measurable function $f$ is strictly level convex then whenever \begin{equation}\label{swMqcx}
\displaystyle{\operatorname*{ess\,sup}_{x\in\Omega} f(\xi_0 + D \varphi(x))= f(\xi_0)},
\end{equation}
for some $\varphi \in W^{1,\infty}_0(\Omega;\mathbb{R}^m)$, $\varphi$ is necessarily $0$. Notice that we include both the scalar and the vectorial case in this assertion.
\end{remark}

\begin{proof}
It is clear that strict level convexity is equivalent to (i). Now we prove that (i) and (ii) are also equivalent. If we assume (i) is true, then it suffices to consider vectors $\xi$ and $\xi+\eta$, the convex combination $\frac{1}{2}\xi+\frac{1}{2} (\xi+\eta)$ and apply (i) with $t=\frac{1}{2}$. Conversely, assume (ii) and let $\xi$ and $\eta$ be such that $f(t\xi+(1-t)\eta)=\max\{f(\xi),f(\eta)\}$ for some $t\in (0,1)$. Suppose, by contradiction, that $\xi\neq\eta$. If $f(\xi)=f(\eta)$ then, since $f$ is level convex, either $f\equiv f(\xi)$ in the segment $[\xi, t\xi+(1-t)\eta]$ or in the segment $[t\xi+(1-t)\eta, \eta]$. But this contradicts (ii). If $f(\xi)\neq f(\eta)$, without loss of generality, we can assume $f(\eta)<f(\xi)$. Then $f\equiv f(\xi)$ in the segment $[\xi, t\xi+(1-t)\eta]$. Indeed, if $f(\zeta)<f(\xi)$ for some $\zeta\in (\xi, t\xi+(1-t)\eta)$ then,
$$f(t\xi+(1-t)\eta)\le \max\{f(\zeta), f(\eta)\}<\max \{f(\xi),f(\eta)\}$$ which is a contradiction. But $f\equiv f(\xi)$ in $[\xi, t\xi+(1-t)\eta]$ also contradicts $(ii)$. This finishes the proof of the equivalence between (i) and (ii).

Concerning condition (iii) first we observe that if $f$ satisfies (iii), then for every $t \in (0,1)$ we can take a function $\varphi$ with value $\xi$ on a set of $\mu $-measure $t$ and $\eta$ ($\eta \not = \xi$) on a set of $\mu$-measure $(1-t)$ to get $f(t\xi+(1-t)\eta)<\max\{f(\xi), f(\eta)\}.$

To prove the viceversa, we argue by contradiction. Assume $f$ is strictly level convex and there exists a nonconstant function $\varphi \in L^1_\mu(\Omega;\mathbb{R}^n)$ such that
$$\displaystyle{f\left(\int_\Omega \varphi\,d \mu\right) =\mu-\operatorname*{ess\,sup}_{x\in\Omega} f(\varphi(x))}= c.$$
If $K = \{\zeta : f(\zeta)\leq c\}$, i.e. $K= L_c(f)$, the level convexity of $f$ guarantees that $K$ is convex and in particular $f(\varphi (x))\leq c$ for $\mu$-a.e. $x \in \Omega$, i.e. $\varphi(x) \in K$ for $\mu$-a.e. $x \in \Omega$. Observe that by Theorem \ref{Danao4.3}, since $\int_\Omega \varphi\, d \mu \in R_c(f)$,  $\int_\Omega \varphi\, d \mu \in Ext L_c (f)$.

 Now if $E= \{x: \varphi (x) \not = \int_\Omega \varphi\, d \mu\}$, since $\varphi$ is assumed nonconstant, $\mu(E)>0$. Moreover, observe that $\frac{1}{\mu(E)}\int_E \varphi\, d \mu =\int_\Omega \varphi \, d\mu$. Clearly for a.e. $x \in E: \varphi(x) \in K':= L_c(f) \setminus \{\int_\Omega \varphi\, d \mu\}$, which  is still a convex set (because it is a convex without an extreme point of it, cf. Theorem \ref{Danao4.3}).
Since $K'$ is convex and $\varphi(x)\in K'$ for a.e. $x \in E$, it results that $\frac{1}{\mu(E)}\int_E \varphi\, d \mu \in K' $,  and this is obviously a contradiction.
\end{proof}


In the remainder of this section we will investigate weaker conditions than strict level convexity.

\begin{definition}\label{slcxi0}
A level convex function $f:\mathbb R^n \to \mathbb R$ is said to be strictly level convex at
$\xi_0\in\mathbb{R}^n$ if for every $t \in (0,1)$ and for every $\xi \not = \eta$: $\xi_0= t\xi + (1-t)\eta$ $\Longrightarrow f(\xi_0) < \max \{f(\xi), f(\eta)\}$.
\end{definition}

With the following result, we introduce a stronger notion than the one in Definition \ref{slcxi0}.

\begin{proposition}\label{endtowardsmid}
Let $f:\mathbb R^n \to \mathbb R$ be a level convex function and let $\xi_0\in\mathbb{R}^n$. Assume that
\begin{equation}\label{strictlevelconvexendpoint}
\text{for every }t \in (0,1)\text{ and for every }\xi\not=\xi_0,\ f(t \xi_0+ (1-t)\xi) < \max \{f(\xi_0), f(\xi)\}.
\end{equation}

Then  $f$ is strictly level convex at $\xi_0$. 
\end{proposition}

\begin{remark}
The reverse implication of Proposition \ref{endtowardsmid}
is not true, to this end it is enough to consider the function $$\xi \in \mathbb R \to f(\xi):= \left\{
\begin{array}{ll}
-\xi & \hbox{ if }\ \xi \leq 0,\\
0 &\hbox{ if }\ \xi >0,
\end{array}
\right.$$
which is strictly level convex at $0$, but it doesn't satisfy condition \eqref{strictlevelconvexendpoint} with $\xi_0=0$.
\end{remark}

\begin{proof}[Proof]
Let $\xi_0= t \xi + (1-t)\eta$ for $t \in (0,1)$ and $\xi, \eta \in \mathbb R^n$.
Let $\theta$ and $\zeta$ be in the segments $(\xi, \xi_0)$ and $(\xi_0, \eta)$ respectively. By \eqref{strictlevelconvexendpoint}
$f(\theta)< \max\{f(\xi_0), f(\xi)\}$ and $f(\zeta)< \max \{f(\xi_0), f(\eta)\}.$
Observe that $\xi_0\in(\theta,\zeta)$.
From the level convexity of $f$ and the previous inequalities it results
$$f(\xi_0)\leq \max\{f(\theta), f(\zeta)\}< \max\{f(\xi), f(\eta), f(\xi_0)\}.$$
Therefore $\max\{f(\xi), f(\eta), f(\xi_0)\}=\max\{f(\xi), f(\eta)\}$ and thus
$$f(\xi_0) <\max \{f(\eta), f(\xi)\}$$
which concludes the proof.
\end{proof}

\begin{proposition}\label{Extxi0slc}
A level convex function $f:\mathbb R^n \to \mathbb R$ is strictly level convex at $\xi_0$ if and only if $\xi_0 \in {\rm Ext}L_{f(\xi_0)}(f)$.
\end{proposition}
\begin{proof}[Proof]
Arguing by contraddiction, assume that $f$ is strictly level convex at $\xi_0$ but $\xi_0  \not \in {\rm Ext}L_{f(\xi_0)}(f)$, namely there exist $\xi $ and $\eta \in L_{f(\xi_0)}(f)$ such that $\xi_0= t \xi+ (1-t)\eta$, $ t \in (0,1)$ and $\xi\neq\eta$, then $\max \{f(\xi), f(\eta)\} > f(\xi_0)$, and this contradicts the fact that $\xi, \eta \in L_{f(\xi_0)}(f)$.

Now we want to prove that if $\xi_0 \in {\rm Ext}L_{f(\xi_0)}(f)$, then $f$ is strict level convex at $\xi_0$. If this was not the case, there would exist $\xi$, $\eta \in \mathbb R^n$, with $\xi, \eta\not= \xi_0$ and $t \in (0,1)$ such that $\xi_0 = t \xi + (1-t)\eta$ and $f(\xi_0)= \max \{f(\xi), f(\eta)\}$ and so $\xi, \eta \in L_{f(\xi_0)}(f)$, which is a contradiction.\end{proof}




We finish this section with a notion, weaker than the strict level convexity at a point either in the sense of Definition \ref{slcxi0} or in the sense of \eqref{strictlevelconvexendpoint}. This will be useful to deal with the minimizing problems in Section \ref{NecSufCond}.

\begin{definition}\label{slc1Dmid}
A level convex function $f:\mathbb R^n \to \mathbb{R}$ is said to be {\it strictly level convex at $\xi_0\in\mathbb{R}^n$ in at least one direction} if there exists $\alpha\in\mathbb{R}^n\setminus \{0\}$ such that: if for some $\gamma$ and $\eta \in \mathbb R^n$
$$\left\{\begin{array}{l} \xi_0=t\gamma+ (1-t)\eta,\ t\in(0,1)\vspace{0.2cm}
\\
f(\xi_0)= \max\{f(\gamma), f(\eta)\} \end{array} \right.$$
then $$
<\gamma-\eta, \alpha>=0.
$$
\end{definition}

\begin{remark}\label{end=midonedir}
One could also give a definition of strict level convexity in at least one direction in the spirit of \eqref{strictlevelconvexendpoint}. Precisely
given $\xi_0\in\mathbb{R}^n$ there exists $\alpha\in\mathbb{R}^n\setminus \{0\}$ such that: if for some $\eta \in \mathbb R^n$

\begin{equation}\label{slcinatleastendpoint}
\left\{\begin{array}{l} \xi=t\xi_0+ (1-t)\eta,\ t\in(0,1)\vspace{0.2cm}\\
f(\xi)= \max\{f(\xi_0), f(\eta)\}\end{array}\right.
\end{equation}
then $$<\xi_0-\eta, \alpha>=0.$$

Then an argument very similar to that employed to prove  Proposition \ref{endtowardsmid} and the fact that $\xi_0, \gamma $ and $\eta$ of Definition \ref{slc1Dmid} are in the same line, guarantee that if $f$ satisfies \eqref{slcinatleastendpoint} at $\xi_0$ then it is also strictly level convex at $\xi_0$ in at least one direction.

On the other hand, the opposite implication is false, that is, there are strictly level convex functions at $\xi_0$ in at least one direction
but not in the sense of \eqref{slcinatleastendpoint}. To this end consider the function $f(\xi)={\rm dist}(\xi,\mathbb{R}^+\times \mathbb{R})$
defined for $\xi\in\mathbb{R}^2$ and take $\xi_0=(0,0)$. Notice also that, this function, although being strictly level convex at $\xi_0$
 in at least one direction, it is not strictly level convex at $\xi_0$.
\end{remark}



We have the following characterization of strict level convexity at a point in at least one direction.

\begin{proposition}\label{strictlcboundary}
A lower semicontinuous and level convex function $f:\mathbb R^n \to \mathbb R$ is strictly level convex at $\xi_0$ in at least one direction if and only if $\xi_0\in \partial L_{f(\xi_0)}(f)$.
\end{proposition}

\begin{remark}
Actually, the lower semicontinuity hypothesis is only required to show that $\xi_0\in \partial L_{f(\xi_0)}(f)$ is a sufficient condition for strict level convexity of $f$ at $\xi_0$ in at least one direction.
\end{remark}

\begin{proof}[Proof] Assume $f$ is strictly level convex at $\xi_0$ in at least one direction. We will show that $\xi_0\in \partial L_{f(\xi_0)}(f)$. Of course $\xi_0\in L_{f(\xi_0)}(f)$. Assume, by contradiction, that $\xi_0 \in {\rm int}L_{f(\xi_0)}(f)$. Then $f(\xi)\le f(\xi_0)$ in a neighborhood of $\xi_0$. Let $\alpha$ be the direction given by Definition \ref{slc1Dmid} and let $\eta_1,\eta_2$ be of the form $\eta_i=\xi_0+\varepsilon_i\alpha$ for some $\varepsilon_i\in\mathbb{R}$ such that $\eta_i$ belong to the neighborhood referred above. Then, by the level convexity of $f$, $f(\xi_0)=\max\{f(\eta_1),f(\eta_2)\}$. But $<\eta_1-\eta_2,\alpha >\neq 0$ which contradicts the hypothesis.

Next we show the opposite implication. Assume that $\xi_0 \in \partial L_{f(\xi_0)}(f)$. The lower semicontinuity of $f$ ensures that $ L_{f(\xi_0)}(f)$ is closed. Notice that, since $\xi_0 \in \partial L_{f(\xi_0)}(f)$, $ L_{f(\xi_0)}(f)\neq\mathbb{R}^n$. Moreover, since $ L_{f(\xi_0)}(f)$ is also convex, there is $\alpha \in\mathbb{R}^n\setminus\{0\}$, such that $<\alpha,\xi_0>\ge <\alpha,\xi>$ for all $\xi\in  L_{f(\xi_0)}(f)$. We will show that $f$ is strictly level convex at $\xi_0$ in the direction $\alpha$. Assume $\xi_0= \lambda \xi + (1-\lambda) \eta$ for some $\lambda\in (0,1)$ and $\eta\neq\xi$ such that $\eta-\xi$ is collinear with $\alpha$. Then at least one of $\xi$ and $\eta$ is not in the set $L_{f(\xi_0)}(f)$, let's say it is $\eta$. Then $f(\eta) > f(\xi_0)$ and thus $f(\xi_0) < \max \{f(\eta), f(\xi)\}$, showing that $f$ is strictly level convex at $\xi_0$ in at least the direction $\alpha$.
\end{proof}

\subsection{Differential inclusions}
We recall that $W^{1,\infty}_0(\Omega)$ is the closure of $C_0^\infty(\Omega)$ in $W^{1,1}(\Omega)$ intersected with $W^{1,\infty}(\Omega)$.

In the sequel we recall two classical results stating necessary and sufficient conditions for existence of solutions to differential inclusions for scalar valued functions. The results are due to Cellina \cite{CellinaNecessary}, \cite{CellinaSufficient}, Friesecke \cite{Friesecke}. See also Bandyopadhyay-Barroso-Dacorogna-Matias \cite{BBDM}. We observe that ${\rm int}\,{\rm co}\,E$ stands for the interior of the convex hull of the set $E$ and we refer respectively to \cite[Theorem 10.24]{D} and \cite[Theorem 2.10]{DaMa} for the proofs.

\begin{theorem}\label{Necdiffincl}
Let $\Omega \subset  R^n$ be a bounded open set, $E \subset \mathbb R^n$, $\xi_0 \in \mathbb R^n$ and denote by $u_{\xi_0}$ an affine function such that $\nabla u_{\xi_0}=\xi_0$. If $u \in u_{\xi_0}+W^{1,\infty}_0(\Omega)$ is such that
$$
\nabla u(x) \in E,\ a.e.\ x \in \Omega,
$$
then
$$
\xi_0 \in E \cup {\rm int}\,{\rm co}\,E.
$$

\end{theorem}

\begin{theorem}\label{Suffdiffincl}
Let $\Omega \subset \mathbb R^n$ be a bounded open set and $E \subset \mathbb R^n$. Let $\varphi\in W^{1,\infty}(\Omega)$ satisfying
\begin{equation}\nonumber
\nabla \varphi(x) \in E \cup {\rm int}\,{\rm co}\,E,\ a.e.\ x \in \Omega.
\end{equation}
Then there exists $u \in \varphi + W^{1,\infty}_0(\Omega)$
such that
$$
\nabla u(x) \in E,\ a.e.\ x \in \Omega.$$
Moreover, given $\varepsilon>0$, $u$ can be chosen such that $||u-\varphi||_{L^\infty(\Omega)}\le\varepsilon$.

\begin{remark}\label{SuffdiffinclRem} The last assertion of the previous theorem follows from a more careful pyramidal construction than the one present in \cite[Theorem 2.10]{DaMa}.
\end{remark}

\end{theorem}
\section{Relaxation Theorem} \label{SectionRelaxation}

Consider the following two minimum problems
\begin{equation}\label{problemP}
(P)\qquad\inf\left\{ \operatorname*{ess\,sup}_{x\in\Omega}f\left(  \nabla u\left(
x\right)  \right):\ u\in u_0+W_{0}^{1,\infty}(\Omega)\right\}
\end{equation}
 and
 \begin{equation}\label{problemPlc}
(P^{\rm lc})\qquad\inf\left\{ \operatorname*{ess\,sup}_{x\in\Omega}f^{\rm lslc}\left(  \nabla u\left(
x\right)  \right):\ u\in u_0+W_{0}^{1,\infty}(\Omega)\right\},
\end{equation}
\noindent where $f$ is given, $f^{\rm lslc}$ is the lower semicontinuous and level convex envelope of $f$, introduced in $(iii)$ of Definition \ref{def1.4}, and $u_0\in W^{1,\infty}(\Omega)$ is the boundary data. Notice that, if $u_0$ is only defined on the boundary of $\Omega$, $u_0:\partial\Omega\longrightarrow\mathbb{R}$, and it is a Lipschitz function, we can extend it to all $\Omega$ and get a $W^{1,\infty}$ function. In this case, and in view of Theorem \ref{NSC} below, the existence of solutions to problem $(P)$ can be ensured depending on the choice of the extension.

The goal of this section is to show that $$\inf (P)=\inf (P^{\rm lc}).$$ This will result as a consequence of the relaxation Theorem \ref{PrinariACV} below.

Before that we recall that level convexity, together with lower semicontinuity is a necessary and sufficient condition for sequential weak* lower semicontinuity in $W^{1,\infty}(\Omega)$ for functionals in the supremal form. The sufficient part is due to Barron-Jensen-Wang cf. \cite[Theorem 3.3]{BJW99}.

\begin{theorem}\label{level convexity sufficient}
Let $f:\mathbb{R}^n\longrightarrow\mathbb{R}$ be a level convex and lower semicontinuous function and let $\Omega\subset\mathbb{R}^n$ be a bounded open set. Then the functional  $\displaystyle F(u)=\operatorname*{ess\,sup}_{x\in\Omega}f( \nabla u(x))$ defined in $W^{1,\infty}(\Omega)$ is sequential weak* lower semicontinuous.
\end{theorem}

The necessary condition follows from \cite[Theorem 3.5]{BJ} (see also \cite[Theorem 4.1]{ABP}).

\begin{theorem}\label{level convexity necessary}
Let $f:\mathbb{R}^n\longrightarrow\mathbb{R}$ be a Borel function and  let $\Omega\subset\mathbb{R}^n$ be a bounded open set. If the functional  $\displaystyle F(u)=\operatorname*{ess\,sup}_{x\in\Omega}f( \nabla u(x))$ defined in $W^{1,\infty}(\Omega)$ is sequentially weak* lower semicontinuous, then $f$ is lower semicontinuous and level convex.
\end{theorem}

Now we establish the relaxation theorem.

\begin{theorem}\label{PrinariACV}
Let $\Omega \subset \mathbb R^n$ be a bounded open set with Lipschitz boundary and let $f: \mathbb R^n \to \mathbb R$ be a continuous function satisfying
\begin{equation}\label{alphacoerciveness}
f(\xi) \geq \gamma (|\xi|),\ \forall\ \xi \in \mathbb R^n,
\end{equation}
with $\gamma :\mathbb R^+ \to \mathbb R^+$ a continuous and increasing function such that $\displaystyle{\lim_{t \to +\infty} \gamma (t)=+ \infty}$.
 Let $u_0 \in W^{1,\infty}(\Omega)$.
Define the functional $F: u \in u_0 + W^{1,\infty}_0(\Omega) \to \displaystyle\operatorname*{ess\,sup}_{x\in\Omega}f( \nabla u(x))$,
and let $\overline{F}$ be  the relaxed functional of $F$ with respect to the weak $\ast$ convergence in $W^{1,\infty}(\Omega)$, namely for every $u\in u_0+ W^{1,\infty}_0(\Omega)$,
\begin{equation}\nonumber
\begin{array}{ll}
\displaystyle{\overline{F}(u)= \inf \left\{\liminf_{h \to + \infty} F(u_h): u_h \in u_0+ W^{1,\infty}_0(\Omega), u_h \rightharpoonup u \text{\,weakly\,} \ast \hbox{\,in }W^{1,\infty}(\Omega)\right\}.}
\end{array}
\end{equation}

\noindent Then
\begin{equation}\nonumber
\displaystyle{\overline{F}(u)= \supess_{x\in \Omega} f^{\rm lslc}(\nabla u(x)),\ \forall\ u \in u_0+ W^{1,\infty}_0(\Omega).}
\end{equation}
\end{theorem}



\begin{remark} A result similar to Theorem \ref{PrinariACV} was proved by Prinari \cite[Theorem 2.6]{Pr09}, with no boundary condition. Our arguments are very similar so we omit them.

In the integral context, this type of results can be proved directly thanks to piecewise affine approximation arguments. In the supremal setting, sets of arbitrarily small measure are determinant to the value of the functional and this kind of arguments is not well fitted. To show the relaxation Theorem \ref{PrinariACV} we need to pass through the sequential weak $\ast$
lower semicontinuous envelope of the functional involved in problem $(P)$, and we extended first this functional to $C(\overline{\Omega})$, as $+ \infty$ in the complement of $u_0+ W^{1,\infty}_0(\Omega)$.  For the reader's convenience we recall that a result devoted to this extension can be found in \cite[Proposition 3.1]{Pr02},  see also \cite[Theorem 8.10 and Corollary 8.12]{DM93}. We also observe that our setting entails that the sequentially weak $\ast$ lower semicontinuous envelope of $f$ coincides with the lower semicontinuous envelope of the extended functional in $C(\overline{\Omega})$, with respect to the uniform tolopology.
\end{remark}

Now we achieve the desired condition $$\inf (P)=\inf(P^{\rm lc}).$$

\begin{corollary}\label{Plc=P}
Let $\Omega \subset \mathbb R^n$ be a bounded open set with Lipschitz boundary and let $f: \mathbb R^n \to \mathbb R$ be a continuous function satisfying condition \eqref{alphacoerciveness}. Let
 $u_0 \in W^{1,\infty}(\Omega)$ and let $(P)$ and $(P^{\rm lc})$ be the problems (\ref{problemP}) and (\ref{problemPlc}), respectively. Then
$$\inf(P)=\inf (P^{\rm lc}).$$

Moreover, if the boundary condition is affine, say $u_0(x)=u_{\xi_0}(x)$ for some $\xi_0 \in \mathbb R^n$, then $$\inf (P)=f^{\rm lslc}(\xi_0).$$
\end{corollary}

\begin{remark}\label{weakquasiconvexity}
From this result we obtain that, for every $\xi_0 \in \mathbb R^n$,
\begin{equation}\label{wMqcxe}
\displaystyle{f^{\rm lslc}(\xi_0)= \inf\left\{ \operatorname*{ess\,sup}_{x\in\Omega}f\left(  \nabla u\left(
x\right)  \right):\ u\in u_{\xi_{0}}+W_{0}^{1,\infty}(\Omega)\right\}.}
\end{equation}
This formula represents an alternative and equivalent formulation for the lower semicontinuous and level convex envelope of $f$,  when $f:\mathbb R^n \to \mathbb{R}$ is continuous and satisfies suitable coerciveness assumptions.

On the other hand \eqref{wMqcxe} is known as weak Morrey quasiconvexity (see \cite[Definition 2.2]{BJW99} when $\Omega$ is a cube). We observe that Corollary \ref{Plc=P} entails that, at least in the scalar case, \eqref{wMqcx} holds in any open set $\Omega$, not necessarily a cube.
Moreover,  we observe that we don't expect formula \eqref{wMqcxe} to provide a characterization to $f^{\rm lslc}$ in the vectorial case. Indeed, the notions of level convexity and weak Morrey quasiconvexity, in general, don't coincide and we don't expect them to coincide even under coercivity assumptions.
\end{remark}

\begin{proof}[Proof] We start proving the first equality.
Clearly $\inf(P) \geq \inf(P^{\rm lc})$. To prove the converse inequality, let $\overline{F}$ be the functional introduced in Theorem \ref{PrinariACV} and observe that the same theorem entails that
$\displaystyle{\overline{F}(u) =\supess_{\Omega}f^{\rm lslc}(\nabla u)}$. Therefore, by the direct method of the calculus of variations, the infimum appearing in $(P^{\rm lc})$ is indeed a minimum, since (\ref{alphacoerciveness}) entails the required coercivity condition and since $\overline{F}$ is sequential weak $\ast$ lower semicontinuous. Thus we can write
$$
\displaystyle{\min(P^{\rm lc}) = \overline{F}(\overline{u}),}
$$
for some $\overline{u}\in u_0+ W^{1,\infty}_0(\Omega)$.

By definition of $\overline{F}$, we also have that there exists a sequence $\{u_n\}\subset u_0+ W^{1,\infty}_0(\Omega)$, such that
$$
\displaystyle{\overline{F}(\overline{u})=\liminf_{n \to +\infty}F(u_n) \geq \inf (P)},
$$
and that proves the first equality in the claim.

To prove the last assertion of the corollary we only need to observe that, if the boundary condition is affine then the supremal Jensen's inequality in Theorem \ref{Jensensupremalscalar}  guarantees that $u_{\xi_0}$ is a solution to problem $(P^{\rm lc})$.
\end{proof}

\section{Necessary and Sufficient Conditions}\label{NecSufCond}

In this section we will investigate necessary and sufficient conditions for existence of solutions to the non-level convex problem $(P)$ introduced in Section \ref{SectionRelaxation}, see equation (\ref{problemP}). Before that we start with some considerations on level convex problems.

Observe that, if $f$ is a level convex function, then the solutions $u$ to the related problem $(P)$ are completely characterized by the following condition
\begin{equation}\label{solutionsforlevelconvexproblems}\left\{\begin{array}{l}u\in u_0+W_0^{1, \infty}(\Omega)\vspace{0.2cm}\\f(\nabla u(x))\le \inf (P),\ a.e.\ x\in\Omega.\end{array}\right.\end{equation}

This characterization shall be compared with Theorem 1 in \cite{CellinaNecessary}. In the present supremal context and for affine boundary condition $u_0=u_{\xi_0}$, the analogous result we should obtain is that for any solution $u$ to problem $(P)$, one has $\nabla u(x)\in F,\ a.e.\ x\in\Omega$ if $F$ is a face of the convex set $L_{f(\xi_0)}(f)$ containing $\xi_0$ in its relative interior. Since the relative interiors of the faces are disjoint, cf. \cite[Theorem 18.2]{Rockafellar}, Proposition \ref{diffincl} below, shows that $L_{f(\xi_0)}(f)$ is the only possible face containing $\xi_0$ in its relative interior. Therefore the analogous result to \cite[Theorem 1]{CellinaNecessary} doesn't give more information than what was stated in condition \eqref{solutionsforlevelconvexproblems}.\medskip

The already mentioned Proposition \ref{diffincl}, concerns uniqueness of solution to a level convex problem.

\begin{proposition}\label{diffincl}
Let $f:\mathbb{R}^n\longrightarrow\mathbb{R}$ be a Borel measurable level convex function and let $(P)$ be the related problem defined in (\ref{problemP}) with an affine boundary condition $u_0=u_{\xi_0}$, $\xi_0\in\mathbb{R}^n$.

If $(P)$ admits a solution $u\neq u_{\xi_0}$, then $\xi_0\in {\rm int}\, L_{f(\xi_0)}(f)$.
\end{proposition}

\begin{proof} Since $f$ is a level convex function, by the supremal Jensen's inequality $\inf (P)=f(\xi_0)$. Therefore, if $u$ is a solution to $(P)$, $u\neq u_{\xi_0}$, then, by \eqref{solutionsforlevelconvexproblems}, $f(\nabla u(x))\le f(\xi_0), \ \forall\ x\in\Omega\setminus A,$ where $A$ is a null measure subset of $\Omega$. Thus  $$S:=\{\nabla u(x):\ x\in\Omega\setminus A\}\subset\{\xi:\ f(\xi)\le f(\xi_0)\},$$ this last set being convex.
On the other hand, since $$\xi_0=\frac{1}{|\Omega|}\int_\Omega \nabla u(x)\,dx,$$ we get $\xi_0\in \overline{\mathrm{co}S}.$
Repeating the argument of Theorem 10.24 in \cite{D}, we get $\xi_0\in \mathrm{int}\,\overline{\mathrm{co}S}=\mathrm{int}\,\mathrm{co}S\subset \mathrm{int}\{\xi:\ f(\xi)\le f(\xi_0)\}.$
\end{proof}

Now, we pass to the case of non-level convex problems. In the same spirit as before, we have the following result. We observe that, in this case, we need to assume the continuity of the function in order to apply the relaxation theorem of the previous section.

\begin{theorem}\label{NSC} Let $\Omega \subset \mathbb R^n$ be a bounded open set with Lipschitz boundary and let
 $f: \mathbb R^n \to \mathbb R$ be a continuous function satisfying (\ref{alphacoerciveness}). Let $u_0 \in W^{1,\infty}(\Omega)$ and let $(P)$ and $(P^{\rm lc})$ be the problems (\ref{problemP}) and (\ref{problemPlc}), respectively.

Then problem $\left(  P\right)$
has a
solution if and only if there exists $u\in u_0+W_{0}^{1,\infty}(\Omega)$ such that
\begin{equation}\label{NSCeq}
f(\nabla u\left(x\right))\le \inf (P^{\rm lc}),\ \text{a.e.}\ x\in \Omega.
\end{equation}
Moreover, if \begin{equation}\label{eqmaisuma}\nabla u_0(x) \in L_{\inf(P^{\rm lc})}(f) \cup {\rm int}\,L_{\inf(P^{\rm lc})}(f^{\rm lslc}),\ \text{a.e.}\ x\in \Omega\end{equation} then $(P)$ has a solution.

In particular, if $u_0$ is affine, say $u_0=u_{\xi_0}$ with $\nabla u_{\xi_0}(x)= \xi_0 \in \mathbb R^n$, then condition \eqref{NSCeq} reads
\begin{equation}\label{NSCeqaffine}
f(\nabla u\left(  x\right))\le f^{\rm lslc}(\xi_0),\ \text{a.e.}\ x\in \Omega.
\end{equation}
Moreover, still under the assumption that $u_0=u_{\xi_0}$ is affine, problem (P) admits a solution if and only if
\begin{equation}\label{bcincluded}
\xi_0 \in L_{f^{\rm lslc}(\xi_0)}(f) \cup {\rm int}\,L_{f^{\rm lslc}(\xi_0)}(f^{\rm lslc}).
\end{equation}
\end{theorem}

\begin{remark}(i) By the supremal Jensen's inequality, the existence of $u\in u_0+W_{0}^{1,\infty}(\Omega)$  satisfying \eqref{NSCeqaffine} is equivalent to the existence of $u\in u_0+W_{0}^{1,\infty}(\Omega)$ such that $$\operatorname*{ess\,sup}_{x\in\Omega}f\left(\nabla u\left(x\right)  \right)=f^{\rm lslc}(\xi_0).$$

(ii) Observe that, for general $u_0$, \eqref{eqmaisuma} is only a sufficient condition, while, for affine functions $u_0$ it is necessary and sufficient (cf. \eqref{bcincluded}).

(iii) If $u_0$ is a Lipschitz function only defined on the boundary of $\Omega$, $u_0:\partial\Omega\longrightarrow\mathbb{R}$, then condition \eqref{eqmaisuma} can be replaced by condition (2.62) in \cite[Theorem 2.17]{DaMa} to get a sufficient condition. Notice that, from \eqref{NSCeq}, we need to solve the differential inclusion $\nabla u\in E$ where $E=\{\xi\in\mathbb{R}^n:\ f(\xi)\le  \inf (P^{\rm lc})\}$.
\end{remark}

\begin{proof} We start proving the first equivalence stated in the theorem. Let $u \in u_0+W^{1,\infty}_0(\Omega) $ be a solution to problem $(P)$. Then, by Corollary \ref{Plc=P}, it results that
$$
\operatorname*{ess\,sup}_{x\in\Omega}f\left(  \nabla u\left(
x\right)  \right)=\inf (P^{\rm lc}),
$$
hence $f(\nabla u(x))\le \inf (P^{\rm lc})$ for a.e. $x \in \Omega$.

To prove the reverse implication, it is enough to observe that $$\inf (P^{\rm lc})\le \inf(P).$$

Then, that condition \eqref{eqmaisuma} is sufficient for existence of solutions to $(P)$, follows from \eqref{NSCeq}, Theorem \ref{Suffdiffincl}, and from Corollary \ref{Corollary to Caratheodory}, where the set $E$ is given by $L_{\inf (P^{\rm lc})}(f)$.

Regarding condition (\ref{NSCeqaffine}), it suffices to observe that, by the supremal Jensen's inequality, cf. Theorem \ref{Jensensupremalscalar}, $f^{\rm lslc}(\xi_0)=\inf (P^{\rm lc})$.

For what concerns the last statement, condition \eqref{bcincluded}, we observe that it follows from Theorems \ref{Necdiffincl} and \ref{Suffdiffincl} and from Corollary \ref{Corollary to Caratheodory}, where the set $E$ is given by $L_{f^{\rm lslc}(\xi_0)}(f)$.
\end{proof}\medskip

As we will see in the next result, in dimension $n=1$, for sufficiently regular $f$, problem $(P)$ with an affine boundary condition, always admits a solution. This is not true for $n>1$ as Example \ref{notalwaysexistence} shows.

\begin{corollary}\label{existencedim1}
Let $\Omega \subset \mathbb R$ be a bounded open set and let $f:\mathbb R \to \mathbb{R}$ be a continuous function satisfying \eqref{alphacoerciveness}.  Then problem $(P)$ introduced in \eqref{problemP} admits a solution for every $u_0:\partial\Omega\longrightarrow\mathbb{R}$.
\end{corollary}

\begin{proof} Let $u_0:\partial\Omega\longrightarrow\mathbb{R}$. Since $\Omega\subset \mathbb{R}$, we can write $\displaystyle\Omega=\cup_{i\in\mathbb{N}}\Omega_i$ with $\Omega_i$ disjoint open intervals. In each of these intervals consider the affine functions $u_{\xi_i}$, for some $\xi_i\in\mathbb{R}$, such that $u_{\xi_i}=u_0$ on $\partial \Omega_i.$ Once proved the existence of solution to each problem
$$(P_i)\quad\inf\left\{ \operatorname*{ess\,sup}_{x\in\Omega_i}f\left( u'(x) \right):\ u\in u_{\xi_i}+W_{0}^{1,\infty}(\Omega_i)\right\},$$ say $u_i\in u_{\xi_i}+W_{0}^{1,\infty}(\Omega_i)$, one gets the existence of solution to $(P)$ patching together the functions $u_i$ in each interval $\Omega_i$.

It remains to prove the existence of solution to each problem $(P_i)$. To achieve this it will be enough to show that every $\xi_i\in \mathbb R$ verifies \eqref{bcincluded} with $\xi_0$ replaced by $\xi_i$. Of course, $\xi_i\in L_{f^{\rm lslc}(\xi_i)}(f^{\rm lslc})=\mathrm{co}\,L_{f^{\rm lslc}(\xi_i)}(f)$, by Corollary \ref{Corollary to Caratheodory}. If $\xi_i \in L_{f^{\rm lslc}(\xi_i)}(f)$, the existence is proven. So, without loss of generality, we may assume that $\xi_i \in {\rm co}\,L_{f^{\rm lslc}(\xi_i)}(f)\setminus L_{f^{\rm lslc}(\xi_i)}(f)$. Since we are working on the real line, we immediately conclude that  $\xi_i \in {\rm int}\,L_{f^{\rm lslc}(\xi_i)}(f^{\rm lslc})={\rm int}\,\mathrm{co}\,L_{f^{\rm lslc}(\xi_i)}(f)$: in $\mathbb{R}$, the elements of the boundary of a convex hull belong either to the original set or to the complement of the convex hull. Thus we proved \eqref{bcincluded} and the proof is finished.
\end{proof}

The following result provides a sufficient condition for existence of solutions to a non-level convex problem with more general boundary data.

\begin{theorem}\label{Theoremexistenceregular}
Let $\Omega \subset \mathbb R^n$ be a bounded open set with Lipschitz boundary and let $f: \mathbb R^n \to \mathbb R$ be a continuous function satisfying (\ref{alphacoerciveness}). Let $u_0 \in  W^{1,\infty}(\Omega)$ and let $(P)$ and $(P^{\rm lc})$ be the problems (\ref{problemP}) and (\ref{problemPlc}), respectively. Assume that problem $(P^{\rm lc})$ admits a solution $\overline{u}\in C^1_{piec}(\Omega)$.

Then, if $f^{\rm lslc}$ is constant in each connected component of the set where $f^{\rm lslc}< f$, problem $(P)$ has a solution.
\end{theorem}

\begin{remark}
The same assertion can be proved under the weaker assumption that $f^{\rm lslc}$ is constant in the connected components of $\{\xi:\ f^{\rm lslc}(\xi)< f(\xi)\}$ whose intersection with $\{\nabla \overline{u}(x):\ x\in \Omega'\}$ is nonempty for some $\Omega'\subset\Omega$ with positive measure.
\end{remark}

\begin{proof} By the continuity hypothesis on $f$ and the coercivity condition (\ref{alphacoerciveness}), Theorem \ref{Caratheodory for f lc} implies that $f^{\rm lslc}$ is also continuous and thus the set
$$A:=\left\{\xi\in\mathbb{R}^n:\ f^{\rm lslc}(\xi)<f(\xi)\right\}$$ is open. Therefore we can split $A$ in a disjoint countable union of open sets $A_i$, $i\in\mathbb{N}$. By hypothesis, in each of these sets $A_i$, the function $f^{\rm lslc}$ is constant. Since the function $\xi\mapsto\gamma(|\xi|)$ is level convex and continuous, $f^{\rm lslc}\ge \gamma(|\cdot|)$. Thus $\displaystyle \lim_{|\xi|\to +\infty}f^{\rm lslc}(\xi)=+\infty$ and this, together with the fact that $f^{\rm lslc}$ is constant in $A_i$ implies that the sets $A_i$ are bounded.

Now, we split the set $\Omega$ in several parts:
$$\begin{array}{l}\displaystyle\Omega_0:=\left\{x\in\Omega:\ \nabla\overline{u}(x)\notin\cup_{i\in\mathbb{N}} A_i\right\}\vspace{0.2cm}\\ \Omega_i:=\left\{x\in\Omega:\ \nabla\overline{u}(x)\in A_i\right\},\ i\in\mathbb{N} \end{array}$$
 and we construct the solution in the following way. For each $i\in\mathbb{N}$ consider $v_i\in \overline{u}+W_0^{1,\infty}(\Omega_i)$ such that  $\nabla v_i\in\partial A_i$, this exists by Theorem \ref{Suffdiffincl} and since $A_i$ being open and bounded, is contained in ${\rm int}\,{\rm co}\,\partial A_i$.

Define $$u(x):=\left\{\begin{array}{l}\overline{u}(x),\text{ if }x\in\Omega_0,\vspace{0.2cm}\\ v_i(x),\text{ if }x\in\Omega_i,\ i\in\mathbb{N}.\end{array}\right.$$ One has $\overline{u}\in u_0+W_0^{1,\infty}(\Omega)$. Moreover, since $f^{\rm lslc}$ is constant on each $A_i$, by the continuity of $f^{\rm lslc}$, it is constant on $\overline{A_i}$. On the other hand, since $A_i$ are open, on their boundary, $f^{\rm lslc}$ coincides with $f$. Therefore, since $\nabla v_i\in\partial A_i$,
$$\displaystyle\operatorname*{ess\,sup}_{x\in\Omega_i}f\left(\nabla v_i\left(
x\right)\right)=\operatorname*{ess\,sup}_{x\in\Omega_i}f^{\rm lslc}\left(\nabla v_i\left(
x\right)\right)=\operatorname*{ess\,sup}_{x\in\Omega_i}f^{\rm lslc}\left(\nabla \overline{u}\left(
x\right)\right)$$
where we have used the fact that $\nabla\overline{u}\in A_i$ in $\Omega _i$.

Therefore $$\begin{array}{l}
\displaystyle\operatorname*{ess\,sup}_{x\in\Omega}f\left(\nabla u\left(
x\right)\right)=\max\left\{\operatorname*{ess\,sup}_{x\in\Omega_0}f\left(\nabla \overline{u}\left(
x\right)\right),\ \operatorname*{ess\,sup}_{x\in\Omega_i}f\left(\nabla v_i\left(
x\right)\right),\ i\in\mathbb{N}  \right\}\vspace{0.2cm}\\
\displaystyle\phantom{\operatorname*{ess\,sup}_{x\in\Omega}f\left(\nabla u\left(
x\right)\right)}=\operatorname*{ess\,sup}_{x\in\Omega}f^{\rm lslc}\left(\nabla \overline{u}\left(
x\right)\right)=\inf(P^{\rm lc})\end{array}$$
and the existence of solution to problem $(P)$ follows from Theorem \ref{NSC}.\end{proof}\medskip

The following example shows many of the features of the results stated in this section.

\begin{example}\label{notalwaysexistence} Let $f:(\xi_1,\xi_2)\in \mathbb R^2 \to (\xi_1^2-1)^2 + \xi_2^2 \in \mathbb R$.
Clearly $f\geq 0$, and $f(\xi_1,\xi_2)= 0$ if and only if $(\xi_1,\xi_2)=\pm (1,0)$. Since $f\ge 0$ then $f^{\rm lslc}\ge 0$, and thus, using Corollary \ref{Corollary to Caratheodory}, $$\left\{f^{\rm lslc}=0\right\}=L_0(f^{\rm lslc})={\rm co}\,L_0(f)=[-1,1]\times \{0\},$$ which has empty interior.

This example also shows that in dimension $n>1$ there are cases where the gradient of the boundary datum doesn't belong neither to $L_{f^{\rm lslc}(\xi_0)}(f)$ nor to ${\rm int }\,L_{f^{\rm lslc}(\xi_0)}(f^{\rm lslc})$
and thus the inclusion \eqref{bcincluded} doesn't admit any solution. By Theorem \ref{NSC}, if $u_0= u_{\xi_0}$ is such that $\nabla u_0(x)\equiv\xi_0 \in (-1,1)\times \{0\}$, considering the minimizing problem $(P)$ related to the function $f$ with the boundary data $u_0= u_{\xi_0}$, we can ensure that $(P)$ doesn't have a solution.

We also observe that, in this example, problem $(P^{\rm lc})$ with affine boundary condition $u_0= u_{\xi_0}$, such that $\nabla u_0(x)\equiv\xi_0 \in (-1,1)\times \{0\}$ has exactly one solution. Indeed, if $u\in u_0+W_0^{1,\infty}(\Omega)$ is a solution to $(P^{\rm lc})$, then $f^{\rm lslc}(\nabla u(x))=0$, a.e. $x\in\Omega$ and thus $\nabla u(x)\in [-1,1]\times \{0\}$, a.e. $x\in\Omega$. In particular, $\nabla u(x)-\xi_0$ is orthogonal to the vector $(0,1)$, a.e. $x\in\Omega$. It then follows by \cite[Lemma 11.17]{D}, that $u\equiv u_{\xi_0}$, showing that $u_{\xi_0}$ is the only solution to $(P^{\rm lc})$.

As we will see in Proposition \ref{Prop no solution}, the fact that $(P^{\rm lc})$ admits a unique solution, and the boundary condition is such that $\xi_0\in\{f^{\rm lslc}<f\}$ implies that $(P)$ has no solution.
\end{example}

We will now turn our attention to necessary conditions for existence of solutions to a non-level convex problem of the form $(P)$ with affine boundary data. Our strategy will be based on uniqueness of solutions to the relaxed problem $(P^{\rm lc})$, as considered for problems in the integral form by Marcellini \cite{Marcellini83}, Dacorogna-Marcellini \cite{DM95}, and Dacorogna-Pisante-Ribeiro \cite{DPR}. The basis to our research is the following result.

\begin{proposition}\label{Prop no solution} Let $\Omega \subset \mathbb R^n$ be a bounded open set with Lipschitz boundary and let $f: \mathbb R^n \to \mathbb R$ be a continuous function satisfying \eqref{alphacoerciveness}. Let $(P)$ and $(P^{\rm lc})$ be the problems \eqref{problemP} and \eqref{problemPlc}, respectively, with the affine boundary condition $u_{\xi_0}(x)=\xi_0\cdot x+c$. Assume that $f^{\rm lslc}(\xi_0)<f(\xi_0)$ and assume $f^{\rm lslc}$ satisfies the condition
$$\left.\begin{array}{c}  \displaystyle\operatorname*{ess\,sup}_{x\in\Omega}f^{\rm lslc}\left(\nabla u\left(x\right)  \right)=f^{\rm
            lslc}(\xi_0)  \vspace{0.2cm}\\
             u\in u_{\xi_0}+W_0^{1,\infty}(\Omega)

\end{array}\right\}\ \Rightarrow\ u=u_{\xi_0},$$
 which means that problem $(P^{\rm lc})$ has a unique solution.

Then problem $(P)$ has no solution.
\end{proposition}

\begin{proof} Since $f^{\rm lslc}\le f$, if $u$ is a solution to $(P)$ then, by Corollary \ref{Plc=P}, it is also a solution to $(P^{\rm lc})$. Therefore, by the uniqueness of solutions to $(P^{\rm lc})$ stated in the hypothesis, $u=u_{\xi_0}$. This contradicts the fact that $f^{\rm lslc}(\xi_0)<f(\xi_0)$ and we conclude that a solution to $(P)$ cannot exist.\end{proof}\bigskip

In view of the result just stated, we want to find conditions ensuring uniqueness of solution to level convex problems of type $(P)$, defined in \eqref{problemP}. We start observing that the strict level convexity of the function related to the minimizing problem $(P)$ provides that uniqueness, if the boundary condition is affine. Indeed this follows from Remark \ref{weakMqcx} and applies to both the scalar and the vectorial cases.
It is interesting to observe that condition \eqref{swMqcx}, involved in Proposition \ref{Prop no solution} is the counterpart in the supremal setting of the notion of strict quasiconvexity at $\xi_0$ which guarantees uniqueness of solutions in the integral framework (see Definition 11.9 and Theorem 11.11 in \cite{D}.)
Therefore, if $f^{\rm lslc}$ was strictly level convex then, for affine boundary conditions, problem $(P^{\rm lc})$, defined in \eqref{problemP}, would have a unique solution and problem $(P)$, defined in \eqref{problemPlc}, would have no solution. However, we show in the next proposition that we can't expect $f^{\rm lslc}\neq f$ to be strictly level convex and thus we will work with a weaker notion.

\begin{proposition}
Let $f:\mathbb{R}^n\to\mathbb{R}$ be a continuous function such that $f^{\rm lc}>-\infty$ and $\displaystyle \lim_{|\xi|\to +\infty}f(\xi)=+\infty$ and let $\xi_0\in\mathbb{R}^n$ be such that $f^{\rm lslc}(\xi_0)<f(\xi_0).$ Then $f^{\rm lslc}$ is constant in a segment line containing $\xi_0$ (possibly $\xi_0$ is an extremity of the segment line).
\end{proposition}

\begin{remark}
The same assertion can be proved if we assume $f$ to be lower semicontinuous, bounded from below, and such that
$\displaystyle \lim_{|\xi|\to +\infty}\frac{f(\xi)}{|\xi|}=+\infty$. The proof in this case is the same that we present below, but one shall
invoke Theorem \ref{thm1.8} (ii), instead of Theorem \ref{Caratheodory for f lc}.
\end{remark}

\begin{proof} By  Theorem \ref{Caratheodory for f lc}, $f^{\rm lslc}(\xi_0)=\max\{f(\xi_1), \dots, f(\xi_{n+1})\}$ for some $\xi_i\in\mathbb{R}^n,\ i=1,\dots, n+1$, such that $\xi_0=\lambda_1\xi_1+\dots + \lambda_{n+1} \xi_{n+1}$ with $\lambda_i\ge 0,\ i=1,\dots, n+1$ and $\sum_{i=1}^{n+1} \lambda_ i = 1$. Moreover, we can assume $\lambda_i> 0$ for every $i=1 ,\dots n+1$ (notice that some $\xi_i$ can be equal) and since $f^{\rm lslc}(\xi_0)<f(\xi_0),$ we conclude that $\xi_0$ belongs to the relative interior of the convex hull ${\rm co}\{\xi_1,\dots,\xi_{n+1}\}$. Therefore we can consider a segment line $[\eta,\zeta]$ contained in this relative interior such that $\xi_0\in (\eta,\zeta)$. By the level convexity of $f^{\rm lslc}$, $$f^{\rm lslc}(\xi)\le \max\{f^{\rm lslc}(\xi_1),\dots,f^{\rm lslc}(\xi_{n+1})\}\le \max\{f(\xi_1),\dots,f(\xi_{n+1})\}= f^{\rm lslc}(\xi_0),\ \forall\ \xi\in [\eta,\zeta].$$
Finally, again by the level convexity of $f^{\rm lslc}$, one has $f^{\rm lslc}\equiv f^{\rm lslc}(\xi_0)$ either in $[\eta,\xi_0]$ or in $[\xi_0,\zeta]$, as wished.\end{proof}

In the spirit of \cite{DM95}, we will deal with a weaker notion of strict level convexity, the strict level convexity in at least one direction, which was introduced in Section \ref{Strict level convexity}. We have the following result.

\begin{theorem}\label{thm strict level conv}
Let $\Omega \subset \mathbb R^n$ be a bounded open set, let $\xi_0\in\mathbb{R}^n$, and let $f:\mathbb{R}^n\longrightarrow\mathbb{R}$ be a Borel measurable and level convex function which is strictly level convex at $\xi_0$ in at least one direction. Then, $u_{\xi_0}$ is the only solution to the problem
\begin{equation}\label{minpb}
\qquad\inf\left\{ \operatorname*{ess\,sup}_{x\in\Omega}f\left(  \nabla u\left(
x\right)  \right):\ u\in u_{\xi_0}+W_{0}^{1,\infty}(\Omega)\right\}.
\end{equation}
\end{theorem}

\begin{proof}
Of course, by the supremal Jensen's inequality, $u_{\xi_0}$ is a solution of the minimizing problem and thus another solution $u$ satisfies $\displaystyle\operatorname*{ess\,sup}_{x\in\Omega}f\left(  \nabla u\left(
x\right)\right)=f(\xi_0)$. Let us fix a representative of $u$ still denoted by $u$.

One has, in particular, $f(Du(x))\le f(\xi_0), \ \forall\ x\in\Omega\setminus A,$ where $A$ is a null measure subset of $\Omega$. Thus  $$S:=\{Du(x):\ x\in\Omega\setminus A\}\subset\{\xi:\ f(\xi)\le f(\xi_0)\}.$$
On the other hand, as in the proof of Proposition \ref{diffincl}
 we get $\xi_0\in \mathrm{int}\,\overline{\mathrm{co}S}=\mathrm{int}\,\mathrm{co}S.$

Now let $x\in\Omega\setminus A$. Then either $Du(x)=\xi_0$ or $Du(x)\neq \xi_0$. If $x$ is in this last case we do the following.
Since $\xi_0\in \mathrm{int}\,\mathrm{co}S,$ we can write $$\xi_0=t\,Du(x)+(1-t)\eta$$ for some $\eta\in \mathrm{co}S$ with $\eta\neq\xi_0$ and $t\in(0,1)$. By the level convexity of $f$
$$f(\xi_0)\le \max\{f(Du(x)), f(\eta)\}\le f(\xi_0).$$ Thus, $\max\{f(Du(x)), f(\eta)\}= f(\xi_0)$ and by the strict level convexity of $f$ at $\xi_0$ in at least one direction one gets $<Du(x)-\eta,\gamma>=0$ for some $\gamma\in\mathbb{R}^n\setminus\{0\}$. Since $Du(x)-\eta$ and $Du(x)-\xi_0$ are colinear, then $<Du(x)-\xi_0,\gamma>=0$.

So we have obtained $<Du(x)-\xi_0,\gamma>=0,\ \forall\ x\in\Omega\setminus A$ and repeating the argument in \cite[Theorem 5.1]{DM95} one gets $u=u_{\xi_0}$ which proves the desired uniqueness of solution.
\end{proof}

\begin{remark}\label{strict lev conv 2}
Another proof for Theorem \ref{thm strict level conv} can be obtained via geometric arguments on the level sets. Namely Proposition \ref{strictlcboundary} ensures that $\xi_0 \in \partial L_{f(\xi_0)}(f)$, and thus by Proposition \ref{diffincl}, \eqref{minpb} admits just the affine solution.

It is worth to observe also that Theorem \ref{thm strict level conv} provides a result analogous to \cite[Proposition 11.14]{D} (where the notion of strict quasiconvexity has been introduced and compared with strict convexity in at least one direction in order to guarantee uniqueness of solutions to integral vectorial minimum problems). In fact  one can read \eqref{swMqcx} as a strict weak Morrey quasiconvexity, and deduce that this condition is weaker than strict convexity at a point in at least one direction.
\end{remark}


Going back to the non-level convex problems we can state the following result which shows that, the strict level convexity at $\xi_0$ in at least one direction for $f^{\rm lslc}$, with $f^{\rm lslc}(\xi_0)< f(\xi_0)$, is the characterizing feature for non existence of solutions to problem $(P)$.

\begin{corollary}\label{corollary slc in at least}
Let $\Omega\subset\mathbb{R}^n$ be a bounded open set with Lipschitz boundary, let $\xi_0\in\mathbb{R}^n$, and let $f:\mathbb{R}^n\longrightarrow\mathbb{R}$ be a continuous function satisfying \eqref{alphacoerciveness}. Assume that $f^{\rm lslc}(\xi_0)< f(\xi_0)$.
Consider problem $(P)$ with $u_0=u_{\xi_0}$. Then problem $(P)$ admits a solution if and only if $f^{\rm lslc}$ is not strictly level convex at $\xi_0$ in any direction.
\end{corollary}

\begin{remark}
Under the assumptions of the corollary, we can state more precisely that, if $f^{\rm lslc}$ is strictly level convex at $\xi_0$ in at least one direction, then problem $(P^{\rm lc})$ has exactly one solution and problem $(P)$ has no solution. This follows from Theorem \ref{thm strict level conv} and Proposition \ref{Prop no solution}, as mentioned in the proof below. If f is also lower semicontinuous, Propositions \ref{diffincl} and \ref{strictlcboundary} ensure that strict level convexity at $\xi_0$ in at least one direction is equivalent to strict weak Morrey quasiconvexity at $\xi_0$.
\end{remark}

\begin{proof}[Proof]
The fact that if $(P)$ has a solution then $f^{\rm lslc}$ is not strictly level convex at $\xi_0$ in any direction is an immediate consequence of Theorem \ref{thm strict level conv} and Proposition \ref{Prop no solution}.

Now we prove the reverse implication. Assume $f^{\rm lslc}$ is not strictly level convex at $\xi_0$ in any direction. Then, by Propostion \ref{strictlcboundary}, $\xi_0 \not \in \partial L_{f^{\rm lslc}(\xi_0)}(f^{\rm lslc})$, being an interior point of this set. Theorem \ref{NSC}, \eqref{bcincluded} ensures then that $(P)$ has a solution.
\end{proof}



Our previous results lead to the following theorem, which intends to characterize the set $R_{f^{\rm lslc}(\xi_0)}(f^{\rm lslc})$ (cf. Notation \ref{CArecalls}) near the point $\xi_0$, for non level convex problems admitting a minimizer and with affine boundary data $u_0=u_{\xi_0}$ .  The result is the analogous version to \cite[Theorem 11.26]{D} for the supremal setting.


\begin{theorem}\label{analog to thm 11.26 Dacorogna}
Let $\Omega\subset\mathbb{R}^n$ be a bounded open set with Lipschitz boundary, let $\xi_0\in\mathbb{R}^n$, and let $f:\mathbb{R}^n\longrightarrow\mathbb{R}$ be a continuous function satisfying \eqref{alphacoerciveness}. Assume that $f^{\rm lslc}(\xi_0)< f(\xi_0)$. Let $K:=\{\xi\in\mathbb{R}^n: f^{\rm lslc}(\xi) < f(\xi)\}$ and assume that $K$ is connected, otherwise we replace $K$ by its connected component containing $\xi_0$. Consider problem $(P)$ with $u_0=u_{\xi_0}$.
\begin{enumerate}
\item[(i)] [Necessary condition.] If $(P)$ has  a minimizer, then there exists $\nu \in \mathbb{R}^n\setminus\{0\}$ and $\e >0$ such that $f^{\rm lslc}$ is constant in the set $\{\xi \in B_\varepsilon(\xi_0):\ <\xi-\xi_0,\nu> \geq 0\}\subset R_{f^{\rm lslc}(\xi_0)}(f^{\rm lslc})$.
\item[(ii)] [Sufficient condition.] If there exists $E \subset \partial K$ such that $\xi_0 \in {\rm int}\,{\rm co}(E)$ and $f^{\rm lslc}\left|_{\{\xi_0\}\cup E}\right.$ is constant then $(P)$ has a solution.
\end{enumerate}
\end{theorem}
\begin{proof}[Proof.]
To prove the necessary part we start observing that, by Corollary \ref{corollary slc in at least}, if $(P)$ admits a minimizer, then $f^{\rm lslc}$ is not strictly level convex at $\xi_0$ in any direction. Then, by Proposition \ref{strictlcboundary}, $\xi_0\in {\rm int} L_{f^{\rm lslc}(\xi_0)}(f^{\rm lslc})$. Let $\varepsilon>0$ be such that $B_\varepsilon(\xi_0)\subset L_{f^{\rm lslc}(\xi_0)}(f^{\rm lslc})$ and consider, for each $n\in\mathbb{N}$, the convex sets $$C_n:=\left\{\xi\in B_\varepsilon(\xi_0):\ f^{\rm lslc}(\xi)\le f^{\rm lslc}(\xi_0)-\frac{1}{n}\right\}.$$ Let $$C:=\bigcup_{n\in\mathbb{N}}C_n=\left\{\xi\in B_\varepsilon(\xi_0):\ f^{\rm lslc}(\xi)< f^{\rm lslc}(\xi_0)\right\}.$$ Observe that $C$, being an increasing sequence of convex sets, it is also a convex set. Moreover, by Theorem \ref{Caratheodory for f lc}, $f^{\rm lslc}$ is continuous and thus $C$ is open.  If $C$ is empty, it means that $f^{\rm lslc}$ is constant in $B_\varepsilon(\xi_0)$. Otherwise, applying a separation result for the convex open set $C$, one gets the existence of $\nu\in\mathbb{R}^n\setminus\{0\}$ such that $<\xi_0-\xi,\nu>< 0$ for all $\xi\in C$. Therefore, for all $\xi\in B_\varepsilon(\xi_0)$ such that $<\xi_0-\xi,\nu>\ge 0$ one has $f^{\rm lslc}(\xi)=f^{\rm lslc}(\xi_0)$, as wished. 

The sufficient part is proved observing that, by Theorem \ref{Suffdiffincl}, there exists $\overline{u} \in u_{\xi_0}+ W^{1,\infty}_0(\Omega)$ such that $\nabla \overline{u} \in E \subset \partial K$ for a.e. $x \in \Omega$. Since $f^{\rm lslc}= f$ on $\partial K$, we have $f(\nabla {\overline u}(x))= f^{\rm lslc}(\nabla {\overline u}(x))$ for a.e. $x \in \Omega$ and since $f^{\rm lslc}$  is constant in $\{\xi_0\}\cup E$, we have
$$
{\rm ess}\sup_{x \in \Omega} f(\nabla {\overline u}(x)) = {\rm ess}\sup_{x \in \Omega} f^{\rm lslc}(\nabla {\overline u}(x))=f^{\rm lslc}(\xi_0),
$$
which, by Theorem \ref{NSC}, ensures that $(P)$ has a solution.
\end{proof}

\begin{remark}\label{fslcnoFslc}
We observe that (as very well emphasized by Crandall \cite{C} and Aronsson-Crandall-Juutinen in \cite{ACJ}) that given $u_0 \in W^{1,\infty}(\Omega)$, and $f:\mathbb R^n \to \mathbb R$ continuous and strict level convex, then the functional
$$
F: u \in C(\overline{\Omega})\to \left\{
\begin{array}{ll}\displaystyle
\supess_{x \in \Omega} f(\nabla u(x)) & \hbox{ if } u \in u_0+ W^{1,\infty}_0(\Omega),\\
\\
+\infty & \hbox{ otherwise,}
\end{array}
\right.
$$
is in general not strictly level convex as a functional, i.e. it does not satisfy $F(u) < \max\{F(u_1), F(u_2)\}$, for every $u= \lambda u_1 + (1-\lambda )u_2$, $u_1, u_2 \in u_0+ W^{1,\infty}_0(\Omega)$, $\lambda\in (0,1)$, even when $n=1$. This is in fact the case of the minimum problem arising when looking for the minimal Lipschitz extension, where in fact the density $f$ defined as $f(\cdot):= |\cdot|$ is strictly level convex, but the minimizer is not unique. On the other hand, our previous results (cf. in particular Remark \ref{weakMqcx},  Proposition \ref{Prop no solution} and Theorem \ref{thm strict level conv}) show (also in the vectorial case) that if $u_0$ is affine, namely $u_0:= u_{\xi_0}$ then we have a unique solution to the problem $\displaystyle\inf\left\{ \supess_{x \in \Omega} f(\nabla u(x)): u \in u_{\xi_0} + W^{1,\infty}_0(\Omega)\right\}$.

We also observe that the continuity and strict level convexity of a function $f: \mathbb R^n \to \mathbb R$, satisfying \eqref{alphacoerciveness}, is sufficient to ensure the uniqueness of solution to the minimization  problem
$$
\inf \left\{\supess_{x \in \Omega} f(u(x)):\ u \in L^\infty(\Omega;\mathbb{R}^n)\right\}.
$$
Moreover, this unique solution is a constant function. Indeed, the coercivity condition and the continuity of $f$ implies the existence of a global minimum to $f$ and the strict level convexity ensures that this minimum is attained in only one point, say $\xi_0\in\mathbb{R}^n$. Therefore $u(x)\equiv \xi_0$ is a solution to the minimization problem and it is obviously the only solution, since any other solution $v$ satisfies $\supess_{x \in \Omega} f( v(x))=f(\xi_0)$ and this implies that $v(x)=\xi_0$  for a.e. $x\in\Omega.$

\end{remark}

\section{Appendix}\label{appendix}

In this section we make some considerations concerning the vectorial case, that is when $f:\mathbb{R}^{m\times n}\to\mathbb{R}$. We also correct some statements by Barron-Jensen-Wang in \cite{BJW99}. In the vectorial case the necessary and sufficient condition for sequential weak * lower semicontinuity of the supremal functional is the so called (strong) Morrey quasiconvexity as proved by Barron-Jensen-Wang \cite[Theorems 2.6 and 2.7]{BJW99}. We start recalling this notion together with other notions also introduced in \cite[Definitions 2.1, 2.2 and 3.7]{BJW99}. We denote by $Q$ the unitary cube of $\mathbb{R}^n$.

\begin{definition}\label{generalized notions}
(i) A Borel measurable function $f:\mathbb R^{m\times n}\to \mathbb R$ is said to be strong Morrey quasiconvex if for  any $\varepsilon>0$, for any $\xi\in \mathbb R^{m\times n}$, and  any $K>0$, there exists a $\delta= \delta(\varepsilon, K, \xi)>0$ such that if $\varphi \in W^{1,\infty}(Q;\mathbb R^m)$ satisfies
$$
\displaystyle{\|\nabla \varphi\|_{L^\infty(Q)} \leq K, \qquad \max_{x \in \partial Q}|\varphi(x)|\leq \delta},
$$
then,
\begin{equation}\nonumber
f(\xi) \leq \operatorname*{ess\,sup}_{x \in Q} f(\xi+ \nabla \varphi(x))+ \varepsilon.
\end{equation}

(ii) A function $f:\mathbb R^{m\times n}\to \mathbb R$ is said to be weak Morrey quasiconvex if
\begin{equation}\label{wMqcx}
f(\xi) \leq \operatorname*{ess\,sup}_{x\in Q}f(\xi+ \nabla \varphi(x)),\end{equation}
for every $\xi\in\mathbb{R}^{m\times n}$ and every $\varphi \in W^{1,\infty}_0(Q;\mathbb R^m)$.

(iii) A function $f:\mathbb R^{m\times n}\to \mathbb R$ is level convex if $f(t\xi+(1-t)\eta)\leq \max \{f(\xi),f(\eta)\}$, for every $t \in [0,1]$ and for every $\xi,\eta \in \mathbb R^{m\times n}$.

(iv) A function $f:\mathbb R^{m\times n}\to \mathbb R$ is rank one quasiconvex (rank one level convex) if for any $\xi,\eta \in \mathbb R^{m\times n}$ with ${\rm rank}(\xi-\eta)\leq 1$, $f(t\xi+(1-t)\eta)\leq \max \{f(\xi),f(\eta)\}$, for every $t \in [0,1]$.
\end{definition}

\begin{remark} Clearly, as observed in \cite{BJW99}, strong Morrey quasiconvexity implies weak Morrey quasiconvexity. However, it is not true that  weak Morrey quasiconvexity implies rank one quasiconvexity, as it was wrongly stated in \cite[Proposition 3.8 and Corollary 3.9]{BJW99}. See Example \ref{example BJW not true}. We will show in Theorem \ref{weakMorreyimpliesrankone} below, that this statement is true if we assume the function to be upper semicontinuous.
\end{remark}

\begin{example}\label{example BJW not true}
Let $m\ge 1$ and $n>1$. Let $S:=\{\xi, \eta\}\subset\mathbb{R}^{m\times n}$ such that $\rm{rank}(\xi-\eta)=1$ and let $f:=1-\chi_S$, where $\chi_S$ is the characteristic function of $S$. Of course $f$ is not rank one quasiconvex. Let's see that $f$ is weak Morrey quasiconvex:
$$f(\zeta)\leq \operatorname*{ess\,sup}_{x\in Q}f(\zeta+ \nabla \varphi(x)),\ \forall\ \zeta\in\mathbb{R}^{m\times n},\ \varphi \in W_0^{1,\infty}(Q;\mathbb{R}^m).$$ To this end it is enough to consider the case where $\zeta \notin S$. Then, the inequality follows from the fact that, there is no $\varphi \in W^{1,\infty}_0(Q;\mathbb{R}^m)$ such that $\nabla \varphi(x)\in \{\xi-\zeta,\eta-\zeta\}$ a.e. in $Q$. Actually, if $m=1$, this is a consequence of Theorem \ref{Necdiffincl}. In the vectorial case $m>1$, this follows from \cite[Propositions 1 and 2]{BallJames}.

We also observe that $f$ is lower semicontinuous. So lower semicontinuity and weak Morrey quasiconvexity is not enough to ensure rank one quasiconvexity.
\end{example}

Next we show that if a function is upper semicontinuous and weak Morrey quasiconvex then it is rank one quasiconvex. We start recalling a lemma due to M{\"u}ller-{\v{S}}ver{\'a}k \cite[Lemma 2.1]{MS}, which is a generalization of a classical one and which will be useful for our proof.
\begin{lemma}\label{key-lemma}Let
$\Omega\subset\mathbb{R}^{n}$ be a bounded open set. Let $t\in\lbrack0,1]$ and
$\xi,\eta\in\mathbb{R}^{m\times n}$
with
${\rm rank}(\xi-\eta)=1$. Let $\varphi$ be an affine map such that
\[
D\varphi(x)=t\xi+(1-t)\eta,\ x\in\overline{\Omega}.
\]
 Then, for every $\varepsilon>0,$ there exists
$u \in Aff_{piec}(\overline{\Omega};\mathbb{R}^m)$ such that
\[
\left\{
\begin{array}
[c]{l}
\mathrm{dist}(Du(x),\{\xi,\eta\})\le \varepsilon,\ \text{a.e. }x\in\Omega,\vspace{0.2cm}\\
\displaystyle\sup_{x\in\Omega}\left\vert u(x)-\varphi(x)\right\vert \leq\varepsilon,\vspace{0.2cm}\\
u(x)=\varphi(x),\ x\in\partial\Omega.
\end{array}
\right.
\]
\end{lemma}

\begin{theorem}\label{weakMorreyimpliesrankone}
(i) Let $f:\mathbb{R}^{m\times n}\to\mathbb{R}$ be an upper semicontinuous and weak Morrey quasiconvex function, then $f$ is rank one quasiconvex. In particular, for $m=1$, if $f$ is upper semicontinuous and weak Morrey quasiconvex, then $f$ is level convex.

(ii) For $m=1$, if $f$ is continuous then $f$ is weak Morrey quasiconvex if and only if $f$ is level convex.

(iii) Let $f:\mathbb{R}^{m\times n}\to\mathbb{R}$ be a Borel measurable function. If $n=1$ and $f$ is weak Morrey quasiconvex then $f$ is level convex.
\end{theorem}

\begin{remark}
Clearly, if a function $f:\mathbb{R}^{m\times n}\longrightarrow\mathbb{R}$ satisfies the supremal Jensen's inequality then it is weak Morrey quasiconvex. The converse being true in the scalar case $n=1$, that is when $\Omega$ is an interval. This follows from the present theorem combined with Theorem \ref{Jensensupremalscalar}.
\end{remark}

\begin{proof}
Once the first assertion of the theorem is proved, the remainder of condition (i) and condition (ii) follow immediately since evidently rank one quasiconvexity reduces to level convexity if $m=1$ and since, level convex functions satisfy the supremal Jensen's inequality (cf. Theorem \ref{Jensensupremalscalar}).

We prove that weak Morrey quasiconvexity implies Morrey rank one quasiconvexity for upper semicontinuous functions. Let $\xi,\eta\in\mathbb{R}^{m\times n}$ be such that ${\rm rank}(\xi-\eta)=1$ and let $t\in(0,1)$. We want to show that $f(t\xi+(1-t)\eta)\le \max\{f(\xi),f(\eta)\}$. Fix $\delta>0$. By the upper semicontinuity of $f$, there exists $\varepsilon>0$ such that $$|X-\xi|\le\varepsilon,\ |Y-\eta|\le\varepsilon\ \Rightarrow\ f(X)\le \delta+f(\xi),\ f(Y)\le \delta+f(\eta).$$
Applying Lemma \ref{key-lemma}, we get $\psi\in W_0^{1,\infty}(Q,\mathbb{R}^m)$ such that
$$\mathrm{dist}(D\psi(x),\{(1-t)(\xi-\eta),-t(\xi-\eta)\})\le \varepsilon,\ \text{a.e. }x\in Q.$$
Using the weak Morrey quasiconvexity and the upper semicontinuity of $f$, we get
$$f(t\xi+(1-t)\eta)\le\operatorname*{ess\,sup}_{x\in Q}f(t\xi+(1-t)\eta+D\psi(x))\le\delta+\max\{f(\xi),f(\eta)\}.$$
The result is now achieved letting $\delta$ go to zero.

It remains to prove condition (iii). Let $f:\mathbb R^{m\times 1} \to \mathbb R$ be a Borel measurable and weak Morrey quasiconvex function, let $\xi,\eta \in \mathbb R^{m\times 1}$, and let $t \in [0,1]$ be arbitrary. Define $$\varphi(x)=\left\{
\begin{array}{ll}(1-t) (\xi-\eta) x &\hbox{ if }0\leq x \leq t,\vspace{0.2cm}\\
t(\xi-\eta)(1-x) &\hbox{ if } t \leq x \leq 1.
\end{array}
\right.$$
Clearly $\varphi \in W^{1,\infty}_0((0,1);\mathbb R^m)$ and applying \eqref{wMqcx} one gets
$ f(t \xi+ (1-t)\eta)\leq \max\{f(\xi), f(\eta)\}$, which gives the level convexity of $f$.
\end{proof}



\section*{Acknowledgements}
The authors are indebted to Giuliano Gargiulo for the discussions on the subject of the paper and  to Professor Luigi Grippo for the suggestions of references, in particular \cite{Danao} and \cite{Mangasarian}.

The work of Ana Margarida Ribeiro was partially supported by the Funda\c{c}\~{a}o para a Ci\^{e}ncia e a Tecnologia (Portuguese Foundation for Science and Technology) through PEst-OE/MAT/UI0297/2011 (CMA), UTA-CMU/MAT/0005/2009 and PTDC/MAT109973/2009.

The work of Elvira Zappale was partially supported by GNAMPA project 2013 `Funzionali supremali: esitenza di minimi e condizioni di semicontinuit\'a nel caso vettoriale', by the Funda\c{c}\~{a}o para a Ci\^{e}ncia e a Tecnologia (Portuguese Foundation for Science and Technology) through UTA-CMU/MAT/0005/2009 and PTDC/MAT109973/2009.

\end{document}